\documentclass[12pt]{amsart}
\usepackage{amssymb, colordvi}
\usepackage{multirow,ulem}
\usepackage{graphicx}

\newtheorem*{VC}{Vakil's Criteria}
\numberwithin{equation}{section}
\newtheorem{theorem}{Theorem}
\newtheorem{proposition}[theorem]{Proposition}
\newtheorem{lemma}[theorem]{Lemma}

\theoremstyle{definition}
\newtheorem{definition}[theorem]{Definition}

\theoremstyle{definition}

\newtheorem{remark}[theorem]{Remark}

\newcounter{FNC}[page]
\def\fauxfootnote#1{{\addtocounter{FNC}{2}$^\fnsymbol{FNC}$%
     \let\thefootnote\relax\footnotetext{$^\fnsymbol{FNC}$#1}}}

\newcommand{\C}{{\mathbb{C}}}
\newcommand{\K}{{\mathbb{K}}}

\renewcommand{\P}{{\mathbb{P}}}

\DeclareMathOperator{\rk}{rk}

\newcommand{\Fdot}{F_{\bullet}}
\newcommand{\Fpdot}{F'_{\bullet}}
\newcommand{\Edot}{E_{\bullet}}
\newcommand{\Gdot}{G_{\bullet}}
\newcommand{\adot}{a_{\bullet}}
\newcommand{\bdot}{b_{\bullet}}
\newcommand{\hdot}{h_{\bullet}}
\newcommand{\eldot}{\ell_{\bullet}}
\newcommand{\Lamdot}{\Lambda_{\bullet}}

\newcommand{\calC}{{\mathcal{C}}}
\newcommand{\calF}{{\mathcal{F}}}
\newcommand{\calG}{{\mathcal{G}}}
\newcommand{\calO}{{\mathcal{O}}}
\newcommand{\calU}{{\mathcal{U}}}
\newcommand{\calW}{{\mathcal{W}}}
\newcommand{\calX}{{\mathcal{X}}}
\newcommand{\calY}{{\mathcal{Y}}}
\newcommand{\calZ}{{\mathcal{Z}}}

\newcommand{\Fl}{{\mathbb{F}\ell}}
\newcommand{\Gr}{\mbox{\rm Gr}}
\newcommand{\pr}{\mbox{\rm pr}}

\newcommand{\codim}{\mbox{\rm codim}}
\newcommand{\blambda}{{\boldsymbol{\lambda}}}
\newcommand{\bmu}{{\boldsymbol{\mu}}}
\newcommand{\bnu}{{\boldsymbol{\nu}}}
\newcommand{\bx}{{\boldsymbol{x}}}
\newcommand{\bS}{{\boldsymbol{S}}}

\newcommand{\lhra}{\ensuremath{\lhook\joinrel\relbar\joinrel\relbar\joinrel\rightarrow}}

\DeclareRobustCommand{\sI}{\includegraphics{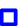}}
\newcommand{\sT}{\includegraphics{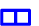}}
\DeclareRobustCommand{\sTh}{\includegraphics{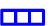}}
\newcommand{\sII}{\includegraphics{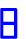}}
\DeclareRobustCommand{\sIII}{\includegraphics{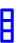}}

\newcommand{\sTI}{\includegraphics{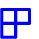}}
\newcommand{\sTII}{\includegraphics{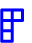}}
\DeclareRobustCommand{\ssTT}{\includegraphics{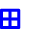}}
\DeclareRobustCommand{\sTT}{\includegraphics{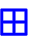}}
\newcommand{\sThI}{\includegraphics{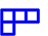}}

\newcommand{\sTTT}{\includegraphics{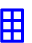}}
\newcommand{\sThTI}{\includegraphics{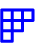}}
\newcommand{\sThTh}{\includegraphics{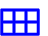}}
\newcommand{\sThTT}{\includegraphics{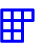}}
\newcommand{\sThThTh}{\includegraphics{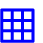}}

\renewcommand{\qed}{\hfill\raisebox{-1.5pt}{\includegraphics[height=9pt]{pictures/s333.eps}}}
\newcommand{\QED}{\hfill\qed}

\newcommand{\defcolor}[1]{\Blue{#1}}
\newcommand{\demph}[1]{\defcolor{{\sl #1}}}

\title[Double transitivity of Galois Groups in Schubert Calculus]{Double transitivity of Galois Groups in Schubert Calculus of Grassmannians}

\author{Frank Sottile}
\address{Frank Sottile \\ Department of Mathematics\\
         Texas A\&M University\\
         College Station\\
         Texas \ 77843\\
         USA}
\email{sottile@math.tamu.edu}
\urladdr{http://www.math.tamu.edu/\~{}sottile/}
\author{Jacob White}
\address{Jacob A. White \\ Department of Mathematics\\
         Texas A\&M University\\
         College Station\\
         Texas \ 77843\\
         USA}
\email{jwhite@math.tamu.edu}
\urladdr{http://www.math.tamu.edu/\~{}jwhite/}

\subjclass{14N15}

\thanks{Research supported in part by NSF grants DMS-0915211 and DMS-1001615.}
\thanks{This material is based upon work supported by the National Science 
Foundation under Grant No. 0932078 000, while Sottile was in 
residence at the Mathematical Science Research Institute (MSRI) in 
Berkeley, California, during the winter semester of 2013.}
\keywords{Schubert problem, Grassmannian, Galois group, double transitivity}
\begin{document}

\begin{abstract}
 We investigate double transitivity of Galois groups in the classical Schubert calculus on
 Grassmannians.
 We show that all Schubert problems on Grassmannians of 2- and 3-planes have doubly
 transitive Galois groups, as do all Schubert problems involving only special Schubert
 conditions.
 We use these results to give a new proof that Schubert problems on Grassmannians of 2-planes
 have Galois groups that contain the alternating group. 
We also investigate the Galois group of every Schubert problem on $\Gr(4,8)$,
   finding that each Galois group either contains the alternating group or  is an 
   imprimitive permutation group and therefore fails to be doubly transitive.
 These imprimitive examples show that our results are the best possible general results on
 double transitivity of Schubert problems.
\end{abstract}
\maketitle

%
\section*{Introduction}
Galois groups are not only symmetry groups of field extensions, they are 
also symmetry groups in enumerative geometry.
This second point was made by Jordan in 1870~\cite{J1870} who studied 
some classical problems in enumerative geometry, showing that several had
Galois groups which were not the full symmetric group, reflecting the intrinsic structure
of these problems.
Earlier, Hermite gave a different connection to geometry, showing that the algebraic Galois
group coincided with a geometric monodromy group~\cite{Hermite}.
Harris used this to study the Galois group of several
problems in enumerative geometry~\cite{Ha79}.
For each he showed that their monodromy groups were the full symmetric groups on their sets
of solutions and therefore the problem had no intrinsic structure.

The Schubert calculus is a well-understood family of problems in enumerative geometry that
involve linear subspaces having prescribed positions with respect to other, fixed linear
spaces.
It provides a laboratory for studying Galois groups in enumerative geometry.
For example, Leykin and Sottile~\cite{LS09} directly computed monodromy for 
many Schubert problems on small Grassmannians involving simple (codimension one) Schubert
conditions and found that each problem had monodromy the full symmetric group.
In~\cite{Va06b}, Vakil gave a general combinatorial method based on the principle of
specialization and group theory for showing that a problem in enumerative geometry has
at least alternating Galois group in that its Galois/monodromy group
contains the alternating group. 
With this method and his geometric Littlewood-Richardson
rule~\cite{Va06a}, Vakil showed that many Schubert problems on small Grassmannians had
at least alternating Galois groups.
He also found Schubert problems whose Galois groups are not the full symmetric group.
Brooks, et al.~\cite{BdCS} used Vakil's combinatorial criterion and some delicate
estimates of integrals to show that all Schubert problems on Grassmannians of 2-planes 
have at least alternating Galois groups.

Vakil gave another, stronger, criterion for showing that a Galois group was at least
alternating which requires knowing that it is a doubly transitive
permutation group.
We study double transitivity of Galois groups of Schubert problems on
Grassmannians.
Vakil's stronger criterion is not our only motivation.
There appears to be a siginificant gap in transitivity: Every Galois group that we
  have studied is either
at least alternating and therefore highly transitive, or else it
fails to be doubly transitive and is imprimitive in that it preserves a partition of the
solutions. 
We conjecture that a Galois group of a Schubert problem is either the full symmetric group on
its solutions, or it is imprimitive.
One purpose of this paper is to give theoretical and computational evidence
supporting this conjecture.

A Schubert condition on $k$-planes in $n$-space is \demph{special} if the condition is
that the $k$-plane meets an $l$-plane nontrivially with $k{+}l\leq n$.
A Schubert problem is \demph{special} if it only has special conditions.
We state our main results.\medskip

\noindent{\bf Theorem.}
 {\it We have the following:
\begin{enumerate}
 \item Every special Schubert problem has a doubly transitive Galois group.
 \item Every Schubert problem in $\Gr(2,n)$ has at least alternating Galois group.
 \item Every Schubert problem in $\Gr(3,n)$ has a doubly transitive Galois group.
 \item There are exactly fourteen Schubert problems in $\Gr(4,8)$ whose Galois groups are
        not at least alternating. 
        For each, the Galois group is imprimitive.
\end{enumerate} 
}\medskip

Part (2) was proven in~\cite{BdCS}. 
Using Part (1) and Vakil's stronger criterion, we 
give a significantly simpler proof of that result. 
This approach does not generalize to show that all 
Schubert problems involving 3-planes have at least
alternating monodromy, for that requires a significantly more delicate analysis of 
geometric problems involving Vakil's checkerboard varieties~\cite{Va06a}.

By part  (4),
our results on double transitivity cannot be extended to Grassmannians of 4-planes. 
One of the fourteen Schubert problems with imprimitive Galois group was first described in
 \S 3.13 of ~\cite{Va06b}, and is due to Derksen.
We determine exactly the Schubert problems on this
Grassmannian whose Galois group is not at least alternating. 
In each case, we compute the Galois group, and demonstrate that the resulting groups are
imprimitive. 
In addition to Derksen's example, there are essentially two others
which generalize to give two families of Schubert problems with imprimitive monodromy.
One family contains a problem in which all but two of its
Schubert conditions are special and the remaining two conditions are dual special
Schubert conditions. 
Hence our results on double transitivity cannot be further extended to simple
general statements. 

This paper is organized as follows.
In Section~\ref{S:MG}, we provide some background on Galois/monodromy groups, and 
explain Vakil's work, including his criteria.
In Section~\ref{S:SC} we give basic definitions in the Schubert calculus on
Grassmannians.
In Section~\ref{S:DT} we discuss how to show double transitivity of Galois groups using
geometry and prove some geometric lemmas.
In Section~\ref{S:MR} we establish
Parts (1) and (3) of our main theorem about special Schubert problems
and Schubert problems involving 3-planes.
In Section~\ref{S:G2n}, we use the double transitivity of special Schubert problems to 
prove Part (2).
We close with Section~\ref{S:IM}, 
in which we prove Part (4) and study the Galois group of every Schubert problem on
  $\Gr(4,8)$.

%
\section{Galois/Monodromy Groups}
\label{S:MG}
We provide some background on Galois/monodromy groups in enumerative geometry. 
More information can be found in~\cite{Va06b}.
We work over an algebraically closed field $\K$ of arbitrary characteristic.
Suppose that $f\colon X\to Y$ is a proper, generically finite and separable (i.e.\
generically \'etale)  
morphism of $\K$-schemes of degree $r$, and $Y$ is irreducible.
Let $\defcolor{X^r}$ be the Zariski closure in the $r$-fold fiber product of the scheme
\[
   \overbrace{X\times_Y X\times_Y \dotsb \times_Y X}^{r}\, \setminus\, \Delta\,,
\]
where $\Delta$ is the big diagonal.
Let $y\in Y(\K)$ be a closed point in the open set over which $f$ is finite and separable
such that the fiber $X_y$ consists of $r$ reduced points.
(Such a point $y$ is a \demph{regular value} of $f$.)
Choose an ordering $\{x_1,\dotsc,x_r\}$ of the points of $X_y$.
Then the \demph{Galois/monodromy group $\calG(X\to Y)$} of $f\colon X\to Y$ is the
subgroup of the symmetric group \defcolor{$S_r$} of permutations of $[r]$
consisting of those permutations $\sigma$ such that the points $(x_1,\dotsc,x_r)$ and  
$(x_{\sigma(1)},x_{\sigma(2)},\dotsc,x_{\sigma(r)})$ in the fiber above $y$ are in the
same irreducible component of $X^r$.
Write \defcolor{$X^r_{\bx}$}, where $\bx:=(x_1,\dotsc,x_r)$, for the component of $X^r$
containing the point $\bx$. 
The group $\calG(X\to Y)$ is well-defined as a subgroup of $S_r$, up to conjugacy.

%
\subsection{Transitivity}
Let $0<t\leq r$ be an integer.
A permutation group $G\subset S_r$ is \demph{$t$-transitive} if any two ordered sets
of $t$ distinct numbers are mapped onto each other by an element of $G$.
We interpret this geometrically for Galois/monodromy groups.
Define $\defcolor{X^t}$ to be the Zariski closure of the scheme
\[
   \overbrace{X\times_Y X\times_Y \dotsb \times_Y X}^{t}\, \setminus\, \Delta\,,
\]
where $\Delta$ is the big diagonal as before.
Let \defcolor{$X^{(t)}$} be the union of
the irreducible components of $X^t$ that map dominantly onto $Y$.
Each component of $X^{(t)}$ has the same dimension as $Y$ and its projection to $Y$ has
finite fibers over the dense subset of regular values of $f$.
The fiber of $X^{(t)}$ over the point $y$ consists of all $t$-tuples $(x'_1,\dotsc,x'_t)$
where $x'_1,\dotsc,x'_t$ are distinct elements of $\{x_1,\dotsc,x_r\}$, taken in every
possible order. 

\begin{lemma}\label{L:t-transitive}
  The Galois/monodromy group $\calG(X\to Y)$ of $f\colon X\to Y$ is $t$-transitive if and
  only if $X^{(t)}$ is irreducible.
\end{lemma}

In particular, $\calG(X\to Y)$ is transitive if and only if $X$ is irreducible.

\begin{proof}
 Let $y\in Y(\K)$ be a regular value of $f$, order the points of $X_y$, and set 
 $\bx:=(x_1,\dotsc,x_r)$.
 Let $X^r_\bx$ be the component of $X^r$ containing $\bx$ and consider its image in $X^t$
 under the projection map $\pi\colon X^r\to X^t$ given by forgetting the last $r{-}t$
 factors in the fiber product.
 As $y$ is a regular value, we have $X^r_\bx\subset X^{(r)}$, and so $\pi(X^r_\bx)$ is an
 irreducible component of $X^{(t)}$.
 The fiber of $\pi(X^r_\bx)$ over $y$ consists of all $t$-tuples
 $(x_{\sigma(1)},\dotsc,x_{\sigma(t)})$ for $\sigma\in\calG(X\to Y)$.

 If $\calG(X\to Y)$ is $t$-transitive, then $\pi(X^r_\bx)=X^{(t)}$, which implies that
 $X^{(t)}$ is irreducible. 
 Conversely, if $X^{(t)}$ is irreducible, then $\pi(X^r_\bx)=X^{(t)}$, which implies that
 $\calG(X\to Y)$ is $t$-transitive. 
\end{proof}

%
\subsection{Vakil's Criteria}
Suppose that we have $f\colon X\to Y$ as above with $X$ irreducible and $Y$ smooth.
Then the Galois/monodromy group $\calG(X\to Y)$ is a transitive subgroup of the symmetric
group $S_r$. 
A subgroup $G$ of $S_r$ is \demph{at least alternating} if it is either the
alternating subgroup of $S_r$ or the full symmetric group $S_r$.
Vakil gave several criteria which may be used to show that $\calG(X\to Y)$ is at least
alternating.
We follow the discussion of~\cite[\S1.1]{BdCS}.

Suppose that we have a fiber diagram
 \begin{equation}\label{Eq:fiber_diagram}
  \raisebox{-22.5pt}{\begin{picture}(64,45)(-3,-1)
   \put(3.5,35){$W$} \put(16,35){$\lhra$} \put(47,35){$X$}
   \put(-3,19){$f$}\put(9,32){\vector(0,-1){20}}
      \put(52,32){\vector(0,-1){20}}\put(54,19){$f$}
   \put(4, 0){$Z$} \put(16, 0){$\lhra$} \put(47, 0){$Y$}
  \end{picture}}
 \end{equation}
where $Z\hookrightarrow Y$ is the closed embedding of a Cartier divisor,  $Y$ 
is smooth in codimension one along $Z$,
and $f\colon W\to Z$ is a generically finite and separable morphism of degree $r$.
When $W$ has at most two components, we have the following.
 \begin{enumerate}
  \item[(i)] If $W$ is irreducible, then there is an inclusion $\calG(W\to Z)$ into
       $\calG(X\to Y)$. 
  \item[(ii)] If $W$ has two components, $W_1$ and $W_2$, each of which maps dominantly to
    $Z$ of respective degrees $r_1$ and $r_2$, then the monodromy group of $f\colon W\to
    Z$ is a subgroup of 
    $\calG(W_1\to Z)\times \calG(W_2\to Z)$ that maps surjectively onto each factor
    $\calG(W_i\to Z)$ and which 
    includes into $\calG(X\to Y)$ (via $S_{r_1}\times S_{r_2}\hookrightarrow S_r$). 
 \end{enumerate}

In~\cite[\S3]{Va06b}, Vakil gave criteria for deducing that $\calG(X\to Y)$ is at least
alternating. 
This follows by  purely group-theoretic arguments including Goursat's Lemma.

\begin{VC}
 Suppose that we have a fiber diagram~$\eqref{Eq:fiber_diagram}$.
 Then $\calG(X\to Y)$ is at least alternating if one of the following holds.
\begin{enumerate}
 \item[{\rm (a)}]   If we are in case $({\rm i})$ and $\calG(W\to Z)$ is at least alternating.
 
 \item[{\rm (b)}]  If we are in case $({\rm ii})$, $\calG(W_1\to Z)$ and $\calG(W_2\to Z)$
                   are at least alternating, and either $r_1\neq r_2$ or $r_1=r_2=1$.

 \item[{\rm (c)}]  If we are in case $({\rm ii})$, $\calG(W_1\to Z)$ and $\calG(W_2\to Z)$
                   are at least alternating, one of $r_1$ or $r_2$ is not $6$, and
                   $\calG(X\to Y)$ is $2$-transitive. 
 \end{enumerate}
\end{VC}

In the Introduction, we referred to Criterion (b) as Vakil's combinatorial criterion,
  and (c) as his stronger criterion.

\begin{remark}\label{R:Vakil}
 These criteria apply to more general inclusions $Z\hookrightarrow Y$ of an irreducible
 variety into $Y$.
 All that is needed is that $Y$ is generically smooth along $Z$, for then we may replace
 $Y$ by an affine open set meeting $Z$ and there are subvarieties
 $Z=Z_0\subset Z_1\subset\dotsb\subset Z_m=X$ with each inclusion $Z_{i-1}\subset Z_i$
 that of a Cartier divisor where $Z_i$ is smooth in codimension one along $Z_{i-1}$, and
 then apply induction.
 \hfill\qed
\end{remark}

Vakil observed that as
the alternating group is least $(r{-}2)$-transitive, a consequence of Criterion
(c) is that to show that a Galois/monodromy group is 
$(r{-}2)$-transitive, it often suffices to show that it is merely doubly transitive.

%
\section{Schubert Calculus}
\label{S:SC}

We
develop the Schubert calculus for the Grassmannian in a form that we will use.
More material, including proofs and references, may be found in~\cite{Fu97}. 
Write \defcolor{$\langle A\rangle$} for the linear span of a subset $A$ of a vector space.
For a positive integer $n$, write \defcolor{$[n]$} for the set $\{1,\dotsc,n\}$.
For a finite-dimensional $\K$-vector space $V$, let \defcolor{$\P(V)$} be the projective
space of 1-dimensional linear subspaces of $V$.
For $0<k<\defcolor{n}:=\dim V$, the \demph{Grassmannian $\Gr(k,V)$} of $k$-dimensional
linear subspaces of $V$ is a smooth irreducible algebraic variety of dimension $k(n{-}k)$.
Then $\Gr(1,V)=\P(V)$.
We will also write $\Gr(k,n)$ for $\Gr(k,V)$.
The Schubert calculus concerns $k$-planes in $V$ having prescribed positions
with respect to other linear subspaces.

The prescribed positions are encoded by \demph{partitions}, which are weakly decreasing
sequences of nonnegative integers,
\[
   \defcolor{\lambda}\ \colon\ n{-}k\geq\lambda_1\geq\lambda_2\geq\dotsb\geq\lambda_k\geq
   0\,.
\]
For example $(4,4,3,1,0)$ is a partition for $\Gr(5,10)$ and $(2,1,0,0)$ is a partition
for $\Gr(4,7)$.
We typically omit trailing 0s and represent a partition by a left-justified
array of boxes with $\lambda_i$ boxes in row $i$. 
Thus these two partitions are
\[
   (4,4,3,1)\ =\ \,\raisebox{-7.5pt}{\includegraphics{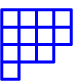}}
   \qquad\mbox{and}\qquad
   (2,1)\ =\ \,\raisebox{-2pt}{\includegraphics{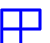}}\ .
\]

A partition prescribes the position of a $k$-plane with respect to a (complete)
\demph{flag} of subspaces, which is a sequence
\[
   \defcolor{\Fdot}\ \colon\ 
     F_1\subsetneq F_2\subsetneq \dotsb \subsetneq F_n\ =\ V\,,
\]
of linear subspaces with $\dim F_i=i$.
Given a partition $\lambda$ and a flag $\Fdot$, the Schubert condition of type $\lambda$
imposed by the flag $\Fdot$ on $k$-planes $H\in\Gr(k,V)$ is 
\[
   \dim H\cap F_{n-k+i-\lambda_i}\geq i\qquad\mbox{for}\quad i=1,\dotsc,k\,.
\]
For the partition $(2,1)$ this is
\[
  \dim H\cap F_{n-k-1}\geq 1
   \qquad\mbox{and}\qquad
  \dim H\cap F_{n-k+1}\geq 2\,,
\]
and the other conditions given by the trailing 0s are that $\dim H\cap F_{n-k+i}\geq i$
for $i\geq 3$, which always hold.

The set of all $H\in\Gr(k,V)$ which satisfy the Schubert condition $\lambda$ on the flag
$\Fdot$ is a \demph{Schubert variety}
 \begin{equation}\label{Eq:GrassSchubVar}
   \defcolor{\Omega_\lambda \Fdot}\ :=\ 
   \{ H\in\Gr(k,V)\,\mid\, \dim H\cap F_{n-k+i-\lambda_i}\geq i
       \mbox{ for }i=1,\dotsc,k\}\,.
 \end{equation}
This is an irreducible subvariety of $\Gr(k,V)$ of codimension 
$\defcolor{|\lambda|}:= \lambda_1+\dotsb+\lambda_k$ and thus of dimension 
$k(n{-}k)-|\lambda|$.
The Schubert variety has a distinguished open subset $\Omega^\circ_\lambda\Fdot$,
which is the locus where the inequalities~\eqref{Eq:GrassSchubVar} are all equalities.

\subsection{Schubert varieties in manifolds of partial flags}
We will also need Schubert varieties in manifolds of partial flags.
A partial flag $\Fdot$ is an increasing sequence of subspaces
\[
   \Fdot\ \colon\ 
    F_{a_1}\subsetneq F_{a_2} \subsetneq\dotsb\subsetneq F_{a_m}\subsetneq V\,,
\]
where $\dim F_{a_i}=a_i$.
This sequence of dimensions $\adot:=(a_1,\dotsc,a_m)$ is the \demph{type} of the
flag $\Fdot$.
The set \demph{$\Fl(\adot,V)$} of all flags of type $\adot$ forms a smooth 
variety of dimension
\[
   \sum_{i=1}^m (n-a_i)(a_i-a_{i-1})\,,
\]
where $a_0$ is taken to be $0$.

This flag variety has distinguished Schubert varieties, which we describe as follows.
Flags $\Fdot$ of type $\adot$ and $\Edot$ of type $\bdot=(b_1,\dotsc,b_l)$ have a
\demph{position}, which is encoded by the \demph{rank array} \defcolor{$\rk(\Fdot,\Edot)$}.
This is the $m\times l$ array where
\[
   \rk(\Fdot,\Edot)_{i,j}\ :=\ \dim F_{a_i}\cap E_{b_j}
    \qquad\mbox{for}\quad i=1,\dotsc,m\mbox{ and } j=1,\dotsc,l\,.
\]
The set of flags $\Fpdot \in\Fl(\adot,V)$ whose position with respect to $\Edot$ is at
least that of $\Fdot$,
 \begin{equation}\label{Eq:SVar}
   X_r(\Edot)\ :=\ \{\Fpdot\in\Fl(\adot,V)\,\mid\, 
      \dim F'_{a_i}\cap E_{b_j}\geq \rk(\Fdot,\Edot)_{i,j} \, ,\ \forall i,j\}\,,
 \end{equation}
is a \demph{Schubert variety}.
This irreducible subvariety of $\Fl(\adot,V)$ has a distinguished dense open
subset $X_r^\circ(\Edot)$ consisting of $\Fpdot$ for which 
$\dim F'_{a_i}\cap E_{b_j}=\rk(\Fdot,\Edot)_{i,j}$ for all $i,j$. 

Only subspaces $F_{n-k+i-\lambda_i}$ with $\lambda_i\neq 0$  in a flag $\Fdot$ are used
to define the Schubert variety $\Omega_\lambda\Fdot$~\eqref{Eq:GrassSchubVar}.
Thus we will often replace complete flags with partial flags of the type occurring
in~\eqref{Eq:GrassSchubVar}, that is 
\[
   \Fdot\ \colon\ 
     F_{n-k+1-\lambda_1}\subsetneq
     F_{n-k+2-\lambda_2}\subsetneq \dotsb \subsetneq
     F_{n-k+m-\lambda_m}\,,
\]
and $\lambda_{m+1}=0$ ($\lambda_m$ is the last nonzero part of $\lambda$.)
Write \defcolor{$\Fl(\lambda,V)$} for this space of partial flags.
It has dimension 
 \begin{multline*}
   \qquad
   \defcolor{N(\lambda)}\ :=\ 
   (n{-}k{+}1{-}\lambda_1)(k{+}\lambda_1{-}1) +
   (\lambda_1{-}\lambda_2{+}1)(k{+}\lambda_2{-}2) \\ 
   + \dotsb +
   (\lambda_{m-1}{-}\lambda_m{+}1)(k{+}\lambda_m{-}m)\,.
  \qquad
 \end{multline*}
For example, when $k=4$ and $n=9$, $N(3,2)=3\cdot 6 + 2\cdot 4 =26$.

When $\lambda=(a,0,\dotsc,0)$ has only one nonzero part, so that it consists of a single
row, we will call it a \demph{special Schubert condition}, and simply write it as $a$.
Dually, when $\lambda$ consists of a single column, so that
$\lambda=(1,\dotsc,1,0,\dotsc,0)$ with $b$ 1s, then it is a \demph{dual special Schubert
  condition}, and we write it as $(1^b)$.

We may omit the subscripts giving the dimensions of subspaces in a flag in
$\Fl(\lambda,V)$. 
Thus for $k=4$ and $n=9$, $\Omega_2K$ means that the partition is $(2,0,0,0)$ and
$K$ has dimension $9{-}4{+}1{-}2=4$, and 
$\Omega_{(3,2)}\Edot$ means that the flag is $E_3\subset E_5$.

Associating a $k$-plane $H$ to its annihilator $H^\perp$ in the dual space $V^*$ of
$V$ gives an isomorphism between $\Gr(k,V)$ and $\Gr(n{-}k,V^*)$, where $n=\dim V$.
Under this isomorphism, $\Omega_\lambda\Fdot$ is sent to 
$\Omega_{\lambda^t}\Fdot^\perp$, where $\lambda^t$ is the transpose of $\lambda$, the
partition obtained by interchanging rows and columns, and $\Fdot^\perp$ is the flag whose
$i$-plane is the annihilator of $F_{n-i}$.
In particular, special Schubert varieties are sent to dual special Schubert
varieties. 

We close this subsection with a dimension calculation.
Let $\adot\colon a_1<\dotsb<a_m$ and $\bdot\colon b_1\leq\dotsb\leq b_m$ be types of
$m$-step flags. 
(Note that we allow $b_i=b_{i+1}$, so that a flag of type $\bdot$ may have repeated
subspaces.) 
For any $\eldot\in\Fl(\bdot,V)$,
consider the Schubert subvariety of $\Fl(\adot,V)$ defined by
\[
   \defcolor{X(\eldot)}\ :=\ 
   \{\Fdot\in\Fl(\adot,V)\ :\ \ell_{b_i}\subset F_{a_i}\ \mbox{for}\ i=1,\dotsc,m\}\,.
\]

\begin{lemma}\label{Lem:Xeldot}
  The Schubert variety $X(\eldot)$ is nonempty if and only if 
 \begin{equation}
  \label{Eq:empty}
    a_j-a_{j-1}\ \geq\ b_j-b_{j-1}\ \mbox{ for }\ j=1,\dotsc,m\,,
 \end{equation}
 where $a_0=b_0=0$.
 When~$\eqref{Eq:empty}$ is satisfied, $X(\eldot)$ has dimension 
 \begin{equation}
  \label{Eq:dimension}
    \sum_{j=1}^m (n-a_j)\bigl((a_j-a_{j-1}) - (b_j-b_{j-1})\bigr)\,.
 \end{equation}
\end{lemma}

Summing~\eqref{Eq:empty} for $j=1,\dotsc,i$ gives the obvious necessary
condition for nonemptiness.

\begin{proof}
 We prove this by induction on the length $m$ of the flags.
 Suppose first that $m=1$.
 Then condition~\eqref{Eq:empty} is $a_1\geq b_1$, and 
 $X(\eldot)=\{F_1\in\Gr(a_1;V) : \ell_1\subset F_1\}$,
 which is simply $\Gr(a_1{-}b_1,V/\ell_1)\simeq\Gr(a_1-b_1,n-b_1)$
 and has dimension $(n{-}a_1)(a_1{-}b_1)$.
 This is nonempty if and only if $a_1\geq b_1$.

 Now suppose that the statement of the lemma holds for flags of length $m{-}1$.
 Given a flag $\Edot$ of length $m$, let 
 $\tau(\Edot):=E_1\subset\dotsb\subset E_{m-1}$ be its truncation to length $m{-}1$ and 
 $\pr(\Edot):=E_m$, its last subspace.
 Set $X'(\tau(\eldot))$ to be those 
 $F_1\subset\dotsb\subset F_{m-1}$ in $X(\tau(\eldot))$ where
 $F_{m-1}\cap\ell_m=\ell_{m-1}$, which is an open subset of $X(\tau(\eldot))$.
 This set is nonempty (and therefore dense) only if there exists a flag $\Fdot$ in
 $X(\tau(\eldot))$ with $F_{m-1}\cap \ell_{m}=\ell_{m-1}$, which is equivalent to the
 inequality $a_{m-1}+b_{m}-b_{m-1}\leq n$.
 But this holds by~\eqref{Eq:empty}, as $a_m\leq n$.

 For $\Fdot'\in X'(\tau(\eldot))$ we consider the fiber
 $\tau^{-1}(\Fdot')\cap X(\eldot)$.
 This is isomorphic to $\pr(\tau^{-1}(\Fdot')\cap X(\eldot))$, which is
\[
   \{ F_m\in\Gr(a_m,V)\ :\ \langle F_{m-1},\ell_m\rangle \subset F_m\}
   \ \simeq\ \Gr(a_m{-}\delta, n{-}\delta)\,,
\]
 where $\delta=\dim \langle F_{m-1},\ell_m\rangle$, which is 
 $a_{m-1}+b_m-b_{m-1}$, as $\Fdot'\in X'(\tau(\eldot))$.
 We must have that $a_m\geq a_{m-1}+b_m-b_{m-1}$ for this to be nonempty, and when this is
 satisfied, we have
\[
   \dim \bigl(\tau^{-1}(\Fdot')\cap X(\eldot)\bigr)\ =\ 
    (n-a_m)[a_m-(a_{m-1}+b_m-b_{m-1})]\,,
\]
 the dimension of this Grassmannian.
 This completes the proof. 
\end{proof}

%
\subsection{Schubert problems on Grassmannians}

A \demph{Schubert problem} on $\Gr(k,n)$ is a list 
$\defcolor{\blambda}=(\lambda^1,\dotsc,\lambda^s)$ of partitions with 
$|\lambda^1|+\dotsb+|\lambda^s|=k(n{-}k)$.
Given a Schubert problem $\blambda$ and a list of general flags $\Fdot^1,\dotsc,\Fdot^s$ with
$\Fdot^i\in\Fl(\lambda^i)$, the intersection 
 \begin{equation}\label{Eq:intersection}
   \Omega_{\lambda^1}\Fdot^1 \,\bigcap\, 
   \Omega_{\lambda^2}\Fdot^2 \,\bigcap\, \dotsb \,\bigcap\, 
   \Omega_{\lambda^s}\Fdot^s \,.
 \end{equation}
is transverse~\cite{K74,Va06a} and consists of finitely many points.
This number \demph{$r(\blambda)$} of points is independent of the choice of general
flags and the algebraically closed field $\K$. 
The Schubert problem is \demph{trivial} if $r(\blambda)=0$. 
By transversality, when the flags are general, all points of the intersection lie in the
intersection of the open Schubert varieties
\[
   \Omega_{\lambda^1}^\circ\Fdot^1 \,\bigcap\, 
   \Omega_{\lambda^2}^\circ\Fdot^2 \,\bigcap\, \dotsb \,\bigcap\, 
   \Omega_{\lambda^s}^\circ\Fdot^s \,,
\]
where $\Omega_\lambda^\circ\Fdot$ is the subset of $\Omega_\lambda\Fdot$ where the
inequalities in~\eqref{Eq:GrassSchubVar} are replaced by equalities.

For a Schubert problem $\blambda=(\lambda^1,\dotsc,\lambda^s)$, consider the family of all
intersections~\eqref{Eq:intersection},
\[
   \defcolor{\calX_{\blambda}}\ :=\ 
   \{ (H,\, \Fdot^1,\dotsc,\Fdot^s)\;:\; \Fdot^i\in\Fl(\lambda^i)\quad\mbox{and}\quad
      H\in \Omega_{\lambda^i}\Fdot^i\quad\mbox{for}\quad i=1,\dotsc,s\}\,.
\]
Set $\calY_{\blambda}:=\prod_i \Fl(\lambda^i)$, the space of $s$-tuples of partial flags for
Schubert varieties indexed by the partitions of $\blambda$.
This family $\calX_{\blambda}$ is equipped with two projections
$\pi\colon\calX_{\blambda}\to\Gr(k,n)$ and 
$f\colon\calX_{\blambda}\to\calY_{\blambda}$. 
Fibers of the projection to $\calY_{\blambda}$ are intersections~\eqref{Eq:intersection},
which, when transverse and nonempty, are zero-dimensional.  
Thus we expect that $\dim \calX_{\blambda}=\dim \calY_{\blambda}$.

In fact, for a Schubert problem $\blambda$ we will always have this equality of dimension.
To see this, define the subvariety $\Psi_\lambda H\subset\Fl(\lambda)$ for $H\in\Gr(k,n)$ to
be 
\[
   \defcolor{\Psi_\lambda H}\ :=\ 
   \{ \Fdot\in\Fl(\lambda)\,\mid\, H\in\Omega_\lambda \Fdot\}\,.
\]
This is a Schubert subvariety of $\Fl(\lambda)$ of codimension $|\lambda|$, and it is
irreducible. 
Now consider fibers of the projection $\defcolor{\pi}\colon \calX_{\blambda}\to\Gr(k,n)$.
For $H\in\Gr(k,V)$, we have
 \begin{equation}\label{Eq:fibre}
  \pi^{-1}(H)\ =\ 
   \Psi_{\lambda^1}H\;\times\;
   \Psi_{\lambda^2}H\;\times\;\dotsb\;\times\;
   \Psi_{\lambda^s}H\,.
 \end{equation}
(Here, as elsewhere, we identify a fiber $\pi^{-1}(H)$ with its image under the second
projection $f$ to $\calY_{\blambda}$.)
Since any two fibers are isomorphic, $\pi\colon\calX_{\blambda}\to\Gr(k,n)$
realizes $\calX_{\blambda}$ as a fiber bundle.
The dimension of a fiber $\pi^{-1}(H)$ is the sum of dimensions of the
Schubert varieties in the product~\eqref{Eq:fibre}, which is 
\[
   \sum_{i=1}^s  (\dim\Fl(\lambda^i)-|\lambda^i|)
      \ =\  \dim\calY_{\blambda}-\dim\Gr(k,n)\,,
\]
from which it follows that $\dim \calX_{\blambda}=\dim \calY_{\blambda}$.
Since the base and fibers of the projection $\pi\colon\calX_{\blambda}\to\Gr(k,n)$
are irreducible, we deduce that $\calX_{\blambda}$ is irreducible.

%
\begin{definition}
 Let $\blambda$ be a Schubert problem.
 By the transversality of a general intersection~\eqref{Eq:intersection}, the map 
 $\calX_{\blambda}\to\calY_{\blambda}$ a proper, generically finite and separable (i.e.\
 generically \'etale) morphism of $\K$-schemes of degree $r(\blambda)$. 
 Furthermore $\calY_{\blambda}$ is irreducible and smooth as it is a product of flag manifolds.
 The \demph{Galois group $\calG(\blambda)$} of the Schubert problem $\blambda$ is the
 Galois group of this family of all instances of the Schubert problem $\blambda$, that
 is, 
\[
    \calG(\blambda)\ :=\ \calG(\calX_{\blambda}\to\calY_{\blambda})\,,
\]
 which acts transitively, as $\calX_{\blambda}$ is irreducible.
\end{definition}

%
\subsection{Galois groups and reduction of Schubert problems}

A Schubert problem is reduced, if, roughly, it is not equivalent to a Schubert problem on a
smaller Grassmannian.
Every Schubert problem $\blambda$ is equivalent to a reduced Schubert problem $\bmu$, and 
we have $\calG(\blambda)\simeq \calG(\bmu)$. Thus, when proving that a Galois group has a given property, 
it suffices to assume that the Schubert problem is reduced.
We give a definition of reduced Schubert problems and sketch a proof of these
facts.

\begin{definition}\label{D:reduced}
 A Schubert problem $\blambda$ on $\Gr(k,n)$ is \demph{reduced} if for every pair of
 partitions $\mu,\nu$ from $\blambda$, none of the following hold.
 \begin{enumerate}
  \item[(a)]  $\mu_1=n{-}k$.
  \item[(b)]  $\mu_k>0$.
  \item[(c)]  There is some $i=1,\dotsc,k$ with $\mu_i+\nu_{k+1-i}\geq n{-}k$.
  \item[(d)]  There is some $i=1,\dotsc,k{-}1$ with $\mu_i+\nu_{k-i}>n{-}k$.
 \end{enumerate}
 If the Schubert problem is nontrivial, then (c) can hold at most with equality, for if 
 $\mu_i+\nu_{k+1-i}> n{-}k$ and the flags $\Edot, \Fdot$ are in general position, then 
 $\Omega_\mu\Edot\cap \Omega_{\nu}\Fdot = \emptyset$.  
\end{definition}

\begin{proposition}\label{Prop:reduced}
 Any nontrivial Schubert problem $\blambda$ is equivalent to a reduced Schubert problem
 $\bmu$ having an isomorphic Galois group, $\calG(\blambda)\simeq\calG(\bmu)$.
\end{proposition}

We indicate how a nonreduced
Schubert problem $\blambda$ is equivalent to a Schubert problem on a smaller Grassmannian.
This process may be iterated to obtain an equivalent reduced Schubert problem.
The arguments given in Remark 4 of~\cite{BdCS} generalize to show that Galois groups
are preserved by this reduction process.

Suppose that $H\in\Omega_\mu\Edot$ in $\Gr(k,v)$ with $\dim V=n$.
Then
\[
  \dim H\cap E_{n-k+1-\mu_1}\geq 1 
   \quad\mbox{and}\quad
  \dim H\cap E_{n-\mu_k}\geq k\,.
\]
Thus if Condition (a) of Definition~\ref{D:reduced} holds, then 
$H\supset E_1$ and if Condition (b) holds, then $H\subset E_{n-\mu_k}\subsetneq V$.
If (a) holds, then $H/E_1\in \Omega_{\mu_{>1}}\Edot/E_1$, which is a Schubert variety in
$\Gr(k{-}1,V/E_1)$, where $\mu_{>1}$ is the partition $\mu_2\geq \mu_3\geq\dotsb\geq\mu_k$
and $\Edot/E_1$ is the image of the flag $\Edot$ in $V/E_1$,
specifically $(\Edot/E_1)_i = E_{i+1}/E_1$.
If $H\in\Omega_\lambda\Fdot$ where $\Fdot$ is a flag in general position
with respect to $\Edot$, then $H/E_1\in\Omega_\lambda \Fdot/E_1$, where $\Fdot/E_1$ is
the image of the flag $\Fdot$ in $V/E_1$, specifically, $(\Fdot/E_1)_i=(F_i+E_1)/E_1$.

If (b) holds, then $H\in\Omega_{\mu-\mu_k}\Edot|_{E_{n-\mu_k}}$, which is a Schubert variety in
$\Gr(k,E_{n-\mu_k})$, where $\mu-\mu_k$ is the partition
\[
   \mu_1-\mu_k\ \geq\ 
   \mu_2-\mu_k\ \geq\ \dotsb\ \geq\ 
   \mu_{k-1}-\mu_k\ \geq\ 0\,,
\]
and $\Edot|_{E_{n-\mu_k}}$ is the flag in $E_{n-\mu_k}$ induced by $\Edot$, specifically,
$E_1\subset E_2\subset \dotsb\subset E_{n-\mu_k}$. 
If $H\in\Omega_\lambda\Fdot$ where $\Fdot$ is a flag in general position
with respect to $\Edot$, then $H\in\Omega_\lambda \Fdot|_{E_{n-\mu_k}}$, where 
$(\Fdot|_{E_{n-\mu_k}})_i=F_{i+\mu_k}\cap E_{n-\mu_k}$.

The point of these reductions is that the original Schubert problem $\blambda$ is
equivalent to a Schubert problem $\bmu$ in a smaller Grassmannian.
All partitions of $\bmu$ are from $\blambda$, except that the partition $\mu$ has been
changed in that either its first row of length $n{-}k$ has been removed or all of its
columns of height $k$ have been removed.

For Condition (c) in Definition~\ref{D:reduced}, note that if
$H\in\Omega_\mu\Edot\cap\Omega_\nu\Fdot$, then for every $i=1,\dotsc,k$, 
 \begin{equation}\label{E:conditionC}
   \dim H\cap E_{n-k+i-\mu_i}\ \geq\ i
    \quad\mbox{and}\quad
   \dim H\cap F_{n-k+(k+1-i)-\nu_{k+1-i}}\ \geq\ k{+}1{-}i\,.
 \end{equation}
Since $H$ has dimension $k$, these conditions imply that
 \begin{equation}\label{E:consequence}
   \dim H \;\cap\; E_{n-k+i-\mu_i}\cap
       F_{n-k+(k+1-i)-\nu_{k+1-i}}\ \geq\  1\,.
 \end{equation}
If $\Edot,\Fdot$ are in general position and $\mu_i+\nu_{k+1-i}>n{-}k$,
then~\eqref{E:consequence} is impossible as
$E_{n-k+i-\mu_i}\cap F_{n-k+(k+1-i)-\nu_{k+1-i}}=\{0\}$.
Thus for the Schubert problem $\blambda$ to have solutions, we must have 
$\mu_i+\nu_{k+1-i}\leq n{-}k$.

Suppose that $\mu_i+\nu_{k+1-i}= n{-}k$.
Then $E_{n-k+i-\mu_i}\cap F_{n-k+(k+1-i)-\nu_{k+1-i}}$ is a 1-dimensional linear space
contained in $H$ that we will call \defcolor{$L$} and $H/L\in\Gr(k{-}1,V/L)$.
Furthermore, if $\mu'$ and $\nu'$ are obtained from $\mu$ and $\nu$ by omitting the parts
$\mu_i$ and $\nu_{k+1-i}$ and $\Edot/L$, $\Fdot/L$ are the images of $\Edot$ and $\Fdot$ in
$V/L$, then $H\in\Omega_{\mu'}\Edot/L\cap \Omega_{\nu'}\Fdot/L$. 
As before, if $H\in\Omega_\lambda\Gdot$ with $\Gdot$ in general position with respect to
$\Edot$ and $\Fdot$, then $H/L\in\Omega_\lambda \Gdot/L$.

Finally, note that if $H\in\Omega_\mu\Edot\cap\Omega_\nu\Fdot$, we
have~\eqref{E:conditionC} as well as $H\cap F_{n-k+(k-i)-\nu_{k-i}}\geq k{-}i$.
If $\Edot$ and $\Fdot$ are in general position, this implies that 
\[
   H\ \subset\ E_{n-k+i-\mu_i} + F_{n-k+(k-i)-\nu_{k-i}}\,.
\]
If we have $\mu_i+\nu_{k-i}>n{-}k$, then this sum $W$ of subspaces of the flags
has codimension $\mu_i+\nu_{k-i}-(n{-}k)$ in $V$.
Then $H\in\Omega_{\mu'}\Edot|_W\cap \Omega_{\nu'}\Fdot|_W$, where $\mu'$ is obtained from
$\mu$ by subtracting $\mu_i+\nu_{k-i}-(n{-}k)$ from each of the first $i$ parts of $\mu$
and $\nu'$ is obtained from $\nu$ by subtracting $\mu_i+\nu_{k-i}-(n{-}k)$ from each of
the first $k{-}i$ parts of $\nu$.

We give an example of this, where $k=3$ and $n=11$ with $\mu=(5,4,0)$ and $\nu=(6,1,0)$.
Then if $i=2$, $k=i=1$ and $\mu_2+\nu_1=10>8=n{-}k$.
Drawing $\mu$ together with $\nu$ (rotated by $180^\circ$) in the $3\times 8$ box shows
that there are two ($=10-8$) columns of height three covered.
Removing those columns from $\mu$ and $\nu$ gives $\mu'=(3,2,0)$ and $\nu'=(4,1,0)$.
\[
   \mu=\raisebox{-4pt}{\includegraphics{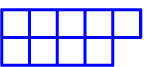}}\qquad
   \nu=\raisebox{-4pt}{\includegraphics{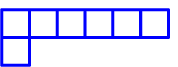}}\qquad\ 
   \raisebox{-6pt}{\includegraphics{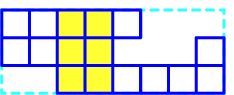}}\qquad\ 
   \mu'=\raisebox{-4pt}{\includegraphics{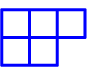}}\qquad
   \nu'=\raisebox{-4pt}{\includegraphics{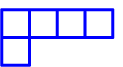}}
\]
As before, if $H\in\Omega_\lambda\Gdot$ in $\Gr(k,V)$ with $\Gdot$ in general position,
then $H\in\Omega_\lambda\Gdot|_W$ in $\Gr(k,W)$.

We also note that any Schubert problem $\blambda$ on $\Gr(k,n)$ is equivalent to a dual
Schubert problem on the dual Grassmannian $\Gr(n{-}k,n)$ which has an isomorphic
Galois group. 

%
\section{Double transitivity in Schubert Calculus}
\label{S:DT}

We develop tools for showing that the Galois/monodromy group $\calG(\blambda)$ of a
Schubert problem $\blambda$ on the Grassmannian $\Gr(k,n)$ is doubly transitive. 
We will assume that $r(\blambda)\geq 2$, for otherwise double transitivity is an empty
condition. 
By Proposition~\ref{Prop:reduced}, we may assume that $\blambda$ is reduced.
By Lemma~\ref{L:t-transitive}, it suffices to show that $\calX^{(2)}_{\blambda}$ is
irreducible.  
For this, we investigate the projection $\pi\colon\calX^2_\blambda\to\Gr(k,n)^2$.
Here, $\calX^2_\blambda$ is a subset of the fiber product $\calX^2_\blambda\times_{\calY_\blambda}\calX^2_\blambda$, but  
$\Gr(k,n)^2$ is simply the ordinary product of Grassmannians.
Unlike the projection $\calX_\blambda\to\Gr(k,n)$, this is not a fiber bundle as
$\pi^{-1}(H_1,H_2)$ depends upon the relative position of $H_1$ and $H_2$.
However, over the locus \defcolor{$\calO_d$} where $\dim H_1\cap H_2=d$ it is a fiber
bundle.
Since $\calX^{(2)}_{\blambda}$ is a union of components of $\calX^2_{\blambda}$ that
project dominantly to $\calY_{\blambda}$, we will study $\pi^{-1}(\calO_d)$ and determine
the components $\calU$ with $\dim\calU=\dim\calY_{\blambda}$. 
From our description of such components, we 
prove Lemma~\ref{Lem:MainLemma}, which gives a numerical condition for $\calU$ 
that is equivalent to $\dim\calU = \dim\calY_{\blambda}$.
In Section~\ref{S:MR} 
we use Lemma~\ref{Lem:MainLemma} to show that $\calX^{(2)}_{\blambda}$ is irreducible 
when $\blambda$ is a special Schubert problem, or a Schubert problem in $\Gr(3,n)$.

We first determine the set-theoretic
fibers $\pi^{-1}(H_1,H_2)$ of $\pi\colon\calX^2_\blambda\to\Gr(k,n)^2$.

\begin{lemma}\label{L:2-fiber}
 Let $\blambda=(\lambda^1,\dotsc,\lambda^s)$ be a Schubert problem on $\Gr(k,n)$ and 
 $(H_1,H_2)\in\Gr(k,n)^2$.
 If $\pi\colon \calX^{2}_\blambda
    =\calX_{\blambda}\times_{\calY_{\blambda}}\calX_{\blambda}\to\Gr(k,n)^2$ is
 the projection, then 
 \begin{equation}\label{Eq:2-fiber}
   \pi^{-1}(H_1,H_2)\ =\ 
      \prod_{i=1}^s \bigl(\Psi_{\lambda^i}H_1\cap \Psi_{\lambda^i}H_2\bigr)\,.
 \end{equation}
\end{lemma}

As before, we are identifying $\pi^{-1}(H_1,H_2)$ with its projection to $\calY_{\blambda}$
in~\eqref{Eq:2-fiber}. 

\begin{proof}
 By the definition of fiber product, 
\[
  \calX^2_\blambda\ =\ 
   \{(H_1,H_2,\Fdot^1,\dotsc,\Fdot^s)\,\mid\,H_1,H_2\in\Omega_{\lambda^i}\Fdot^i\ 
      \mbox{ for }i=1,\dotsc,s\}\,.
\]
 Since $H\in\Omega_\lambda\Fdot$ if and only if $\Fdot\in \Psi_\lambda H$, we see that
\[
  \pi^{-1}(H_1,H_2)\ =\ \{(\Fdot^1,\dotsc,\Fdot^s)\,\mid\,
     \Fdot^i\in \Psi_{\lambda^i}H_1\cap \Psi_{\lambda^i}H_2\ 
      \mbox{ for }i=1,\dotsc,s\}\,,
\]
 which implies~\eqref{Eq:2-fiber}.
\end{proof}

\begin{remark}\label{R:dominant_components}
 A component $\calU$ of $\pi^{-1}(\calO_d)$ that projects dominantly to
 $\calY_{\blambda}$ must contain points $(H_1,H_2,\Fdot^1,\dotsc,\Fdot^s)$ where the flags 
 $(\Fdot^1,\dotsc,\Fdot^s)$ are in general position so that 
 $H_1,H_2\in\Omega^\circ_{\lambda^i}\Fdot^i$.
 This implies that $\Fdot^i\in \Psi^\circ_{\lambda^i}H_1\cap \Psi^\circ_{\lambda^i}H_2$.
 Consequently, in our arguments we may replace 
 $\Psi_\lambda H$ by any subset containing $\Psi^\circ_\lambda H$.
 \QED
\end{remark}

By Lemma~\ref{L:2-fiber}, we must first understand the set-theoretic
intersection $\Psi_\lambda H_1 \cap \Psi_\lambda H_2$ when 
$(H_1,H_2)\in\calO_d$.
We begin with the case when the partition is a special Schubert condition,
$\lambda=(a)$, abbreviated \defcolor{$a$}. 
Flags of type $a$ are elements $K$ of $\Gr(n{-}k{+}1{-}a,n)$.
For $H$ of dimension $k$ and $K$ of dimension $n{-}k{+}1{-}a$ write
$\Omega_aK\subset\Gr(k,n)$ and $\Psi_a H\subset\Gr(n{-}k{+}1{-}a,n)$ for the corresponding
special Schubert varieties,
 \begin{eqnarray*}
   \Omega_aK&=& \{H'\in\Gr(k,n)\,\mid\,H'\cap K\neq \{0\}\}\,,\ \ \mbox{ and}\\
   \Psi_a H&=&\{K'\in\Gr(n{-}k{+}1{-}a,n)\,\mid\,H\cap K'\neq \{0\}\}\,,
 \end{eqnarray*}
and note that $\Psi_a H=\Omega_a H$.
Set $\defcolor{N(a)}:=\dim\Gr(n{-}k{+}1{-}a,n)=(n{-}k{+}1{-}a)(k{-}1{+}a)$, so that
$\dim \Psi_a H=N(a)-a$.
We expect that $\dim \Psi_a H_1\cap \Psi_a H_2=N(a)-2a$.

\begin{lemma}
\label{Lem:SpecComponents}
 Let $(H_1,H_2)\in\calO_d$ with $d<k$.
 If $a{+}d\leq k$ or $d=0$ then the intersection $\Psi_{a}H_1 \cap \Psi_{a}H_2$ has dimension
 $N(a){-}2a$ and it is irreducible if $a{+}d<k$ or $d=0$.
 When $a{+}d\geq k$ and $d>0$, it consists of two components, one of dimension $N(a){-}2a$, 
 and the other of dimension $N(a){-}a{-}(k{-}d)$.
 This second, possibly excess, component consists of those $K$ with 
 $K\cap H_1\cap H_2\neq\{0\}$.   
\end{lemma}

\begin{proof}
 Set $\defcolor{L}:=H_1\cap H_2$, which is $d$-dimensional.
 Since $\dim \langle H_1,H_2\rangle= 2k{-}d\leq n$, we have that 
 $\max\{0,2k{-}n\}\leq d(<k)$.
 The intersection is a disjoint union
 \begin{equation}\label{Eq:special_decomp}
   \Psi_a H_1\cap \Psi_a H_2\ =\ U_0 \sqcup U_1\,,
 \end{equation}
 where \defcolor{$U_0$} consists of those $K\in \Psi_a H_1\cap \Psi_a H_2$ with 
 $K\cap L=\{0\}$ and    
 \defcolor{$U_1$} consists of those $K$ with $\dim(K\cap L)\geq 1$.
 We compute the dimensions of $U_0$ and $U_1$ by studying incidence varieties which map
 birationally to each.

 Set 
 $\defcolor{\widetilde{U}_0}:=\{(K,h_1,h_2)\,\mid\, K\in U_0, h_i\subset K\cap H_i, \dim h_i=1\}$.
 This projects birationally to $U_0$, and its image in the last two factors 
 $\widetilde{U}_0\to\P(H_1)\times \P(H_2)$ is the open subset 
\[
    \bigl(\P(H_1)\smallsetminus \P(L)\bigr)\times
    \bigl(\P(H_2)\smallsetminus \P(L)\bigr)
\]
 with fiber over $(h_1,h_2)$ those $K$ which contain 
 the 2-dimensional linear span $\langle h_1,h_2\rangle$ but do not meet
 $L$.
 This is a Zariski open subset of the Grassmannian
\[
\Gr(n{-}k{-}a{-}1,\K^n/\langle h_1,h_2\rangle)\simeq\Gr(n{-}k{-}a{-}1,n{-}2)\,.
\]
 Thus the dimension of $\widetilde{U}_0$ (and thus of $U_0$) is 
 \begin{multline*}
  \qquad
  \dim\P(H_1)+\dim\P(H_2) + \dim \Gr(n{-}k{-}a{-}1,n{-}2)\\ =\ 
   2(k{-}1)+(n{-}k{-}a{-}1)(a{+}k{-}1)\ =\ 
   N(a){-}2a\,.
  \qquad
 \end{multline*}

 For $U_1$, define 
 $\defcolor{\widetilde{U}_1}:=\{(K,\ell)\,\mid\, K\in U_1,\ell\subset K\cap L,\dim\ell=1\}$.
 This projects birationally to $U_1$.
 The second projection 
 $\widetilde{U}_1\to\P(L)$ is surjective with fiber over $\ell\in\P(L)$
 the set of those $K$ which contain $\ell$, which is
 $\Gr(n{-}k{-}a,\K^n/\ell)\simeq\Gr(n{-}k{-}a,n{-}1)$.
 Thus the dimension of $\widetilde{U}_1$ (and thus of $U_1$) is 
\[
  \dim\P(L)+\dim \Gr(n{-}k{-}a,n{-}1)\ =\ 
   d{-}1+(n{-}k{-}a)(a{+}k{-}1) \ =\ 
   N(a){-}a{-}(k{-}d)\,.
\]
  Note that $\dim U_0=N(a)-2a$, the expected dimension of 
  $\Psi_a H_1\cap \Psi_a H_2$.
  When $d=0$, $U_1=\emptyset$, and this intersection is irreducible.
  If $a+d<k$, then $\dim U_1<N(a)-2a$ and so $U_1$ lies in the closure of $U_0$.
  When $a+d\geq k$, $\dim U_1\geq \dim U_0$, and so its closure is a component of the
  intersection, which has the expected dimension $N(a)-2a$ only when $a+d=k$.
\end{proof}

We introduce notation for the sets $U_0$ and $U_1$ in the
decomposition~\eqref{Eq:special_decomp}.
Write \defcolor{$U_0(a,H_1,H_2)$} for the set $U_0$ and 
\defcolor{$U_1(a,H_1,H_2)$} for the set $U_1$.

We now consider intersections $\Psi^\circ_{\lambda}H_1\cap \Psi^\circ_{\lambda}H_2$, where 
$\lambda$ is any Schubert condition.
Let $(H_1,H_2)\in\calO_d\subset\Gr(k,n)^2$ with $d<k$ and set $L:=H_1\cap H_2$, so that
$\dim L=d$.
Since we assume that our Schubert problem is reduced, we will have $\lambda_1<n{-}k$ and 
$\lambda_m>0=\lambda_{m+1}$ for some $m<k$.  
We decompose $\Psi_{\lambda}^\circ H_1 \cap \Psi_{\lambda}^\circ H_2 $ into $2^m$ disjoint
subsets and then compute the dimension of those that are nonempty.

\begin{definition}\label{def:U_S}
  For $S \subset [m]$, define $\defcolor{U_S}=\defcolor{U_S(\lambda,H_1,H_2)}$ to be 
\[
  \{\Fdot\in \Psi^\circ_{\lambda}H_1\cap \Psi^\circ_{\lambda}H_2\ :\
           \dim F_{n-k+j-\lambda_j} \cap L = |[j] \cap S|\,,\ j=1,\dotsc,m \}\,.
\]
  Recall that $[j]:=\{1,\dotsc,j\}$.
  Observe that we must have $|S|\leq d=\dim L$, and that 
\[
    \Psi^\circ_{\lambda}H_1\cap \Psi^\circ_{\lambda}H_2\ =\ 
    \bigsqcup_{S} U_S(\lambda,H_1,H_2)\,.
\]
 Indeed, for each $j=1,\dotsc,m$ and $i=1,2$, we have $\dim F_{n-k+j-\lambda_j}\cap H_i=j$, so
 that $\dim F_{n-k+j-\lambda_j}\cap L\leq j$.
 Furthermore, if $M_j:=F_{n-k+j-\lambda_j}\cap L$, then 
 $M_1\subset M_2\subset\dotsb\subset M_m$. 
 All possibilities for the dimensions of $M_j$ are given by $|[j]\cap S|$ for subsets $S$
 of $[m]$ of cardinality at most $d$.
\end{definition}

\begin{lemma} \label{Lem:Components}
 When $U_S \neq \emptyset$, the set $U_S$ is irreducible of dimension 
 \[
   N(\lambda) - 2|\lambda| + 
     \sum_{j \in S} \left(\lambda_j - k + d + |[j] \smallsetminus S|\right)\,.
 \]
\end{lemma}

\begin{proof}
 The intersections of a flag $\Fdot\in U_S$ with  $H_1$, $H_2$, and $L$ are flags 
 $\hdot^1, \hdot^2, \eldot$ in these spaces with prescribed incidences.
 Specifically, for each $i=1,2$ and $j=1,\dotsc,m$, we have 
\[
    h^i_j\ :=\ F_{n-k+j-\lambda_j}\cap H_i
   \qquad\mbox{and}\qquad
   h^1_j \cap L\ =\ 
   h^2_j \cap L\ =\ 
   \ell_{|[j] \cap S|}\ :=\
   F_{n-k+j-\lambda_j} \cap L\,.
\]
 Then we have 
\[
   \hdot^1\in\Fl(1, \ldots, m, H_1)\,,\ 
   \hdot^2\in\Fl(1, \ldots, m, H_2)\,,\ \mbox{and}\ 
   \eldot\in\Fl(1,\ldots,|S|,L)\,.
\]
 The flags $\hdot^i\subset H_i$ instantiate the conditions for $\Fdot$ to belong to 
 $\Psi_{\lambda}^\circ H_i $ and the flag $\eldot$ measures how the flags $\hdot^1$,
 $\hdot^2$, and $\Fdot$ meet $L=H_1\cap H_2$.

 Let \defcolor{$B_S$} consist of triples $(\hdot^1, \hdot^2,\eldot)$ of flags, where
\[
   \hdot^i\in\Fl(1, \ldots, m; H_i)\,,\ \mbox{for}\ i=1,2\,,\ \ 
   \eldot\in\Fl(1,\ldots,|S|;L)\,,\ \mbox{ and }\  
   h^i_j \cap L\ =\ \ell_{|[j] \cap S|}\,.
\]
 Mapping $\Fdot\in U_S$ to its intersections $(\hdot^1,\hdot^2,\eldot)$ with $H_1$,
 $H_2$, and $L$ gives a projection $\varphi\colon U_S \to B_S$. 
 We compute the dimension of the fiber over $(\hdot^1, \hdot^2, \eldot) \in B_S$. 

 Set $\defcolor{\Lambda_j}:= \langle h^1_j, h^2_j \rangle$, which has dimension 
 $2j-|[j]\cap S|$.
 Then the fiber $\varphi^{-1}(\hdot^1, \hdot^2,\eldot)$ is identified with its projection
 to $\Fl(\lambda,V)$ which is an open subset of the Schubert variety $X(\Lamdot)$ of
 Lemma~\ref{Lem:Xeldot}.  
 In the notation of that lemma, $a_j=n{-}k{+}j{-}\lambda_j$ and 
 $b_j=2j-|[j]\cap S|$.
 Then the fiber $\pi^{-1}(\hdot^1, \hdot^2,\eldot)$ has dimension
 \begin{eqnarray*}
   N(\lambda)\ -\ \sum_{j=1}^m (k{-}j{+}\lambda_j)(b_j{-}b_{j-1})
   &=&
   N(\lambda)\ -\ 2\sum_{j=1}^m (k{-}j{+}\lambda_j)
             \ +\ \sum_{j\in S} (k{-}j{+}\lambda_j)\\
   &=&
   N(\lambda)\ -\ 2|\lambda|\ -\ 2\sum_{j=1}^m (k{-}j)
             \ +\ \sum_{j\in S} (k{-}j{+}\lambda_j)\,,
 \end{eqnarray*}
 as $b_{j}=b_{j-1}+1$ if $j\in S$ and $b_{j}=b_{j-1}+2$ if $j\not\in S$.

We compute the dimension of $B_S$. 
Let $\defcolor{\rho}\colon B_S \to \Fl(1,\ldots, |S|; L)$ be the projection forgetting
$\hdot^1$ and $\hdot^2$. 
For $\eldot \in \Fl(1, \ldots, |S|; L)$, the fiber $\rho^{-1}(\eldot)$ is an
open dense subset of
\[
   X(\eldot)\times X(\eldot)\ \subset\ 
  \Fl(1,\ldots, m; H_1)\times \Fl(1,\ldots, m; H_2)\,,
\]
where $X(\eldot)$ is the  Schubert variety of Lemma~\ref{Lem:Xeldot}.
Since $\Fl(1,\ldots, m; H_i)$ has dimension $N(1,\dotsc,m)=\sum_{j=1}^m (k-j)$
and in the notation of Lemma~\ref{Lem:Xeldot}, $a_j=j$ and $b_j=|[j]\cap S|$, this 
fiber $\rho^{-1}(\eldot)$ has dimension  
 \[
   2\bigl( N(1,\dotsc,m)\ -\ \sum_{j=1}^m(k-j)(b_j-b_{j-1})\bigr)\ =\ 
   2\sum_{j=1}^m (k-j)\ -\ 2\sum_{j\in S}(k-j)\,.
\]
Finally, as $\Fl(1,\ldots, |S|; L)$ has dimension
\[
  \sum_{i=1}^{|S|} (d-i)\ =\ 
  \sum_{j\in S} (d - |[j] \cap S|)\ =\ 
  \sum_{j\in S} (d-j+|[j]\smallsetminus S|)\,,
\]
 we have
\[
   \dim U_S\ =\ 
  N(\lambda) - 2|\lambda| + \sum_{j \in S} (\lambda_j - k + d + |[j] \smallsetminus S|)\,.
\]
Since all sets and fibers that we considered are irreducible, $U_S$ is irreducible.
\end{proof}

The decomposition of $\Psi^\circ_\lambda H_1\cap \Psi^\circ_\lambda H_2$ in
Definition~\ref{def:U_S} as a disjoint union of sets $U_S(\lambda,H_1,H_2)$ gives a
decomposition of the 
product of the intersections $\Psi^\circ_{\lambda^i}H_1\cap \Psi^\circ_{\lambda^i}H_2$.
For each $i=1,\dotsc,s$, let $m_i$ be the index of the last nonzero part of $\lambda^i$.
For each $d=0,\dotsc,k{-}1$ and each sequence $\bS=(S^1,\dotsc,S^s)$ where $S^i$ is a
subset of $[m_i]$ of cardinality at most $d$, define 
\[
   \defcolor{C_{\bS}(H_1,H_2)}\ :=\ 
   \prod_{i=1}^s U_{S^i}(\lambda^i,H_1,H_2)\,.
\]
We write the product of the intersections
$\Psi^\circ_{\lambda^i}H_1\cap \Psi^\circ_{\lambda^i}H_2$ as a disjoint union
\[
   \prod_{i=1}^s \Psi^\circ_{\lambda^i}H_1\cap \Psi^\circ_{\lambda^i}H_2
    \ =\ 
   \bigsqcup_{\bS}C_{\bS}(H_1,H_2)\,,
\]
where $\bS$ ranges over all such sequences of subsets.

For each choice of $\bS$, the sets $C_{\bS}(H_1,H_2)$ for $(H_1,H_2)\in\calO_d$ form
a fiber bundle over $\calO_d$, which we write as \defcolor{$\calC_{\bS,d}$}.
The following proposition is clear.

\begin{proposition}\label{P:XtwoComponents}
  $\calX^{(2)}_{\blambda}$ is a subset of the closure of\/ 
 ${\displaystyle \bigsqcup_{\bS,d}\calC_{\bS,d}}$.
\end{proposition}

Only sets $\calC_{\bS,d}$ of dimension equal to that of $\calY_{\blambda}$ could 
map dominantly to $\calY_{\blambda}$ with finite fibers over an open subset, and so it 
is a necessary condition that 
\[
   \dim \calC_{\bS,d}\ =\ \dim \calY_{\blambda}
\]
for the closure of $\calC_{\bS,d}$ to be a component of $\calX^{(2)}_{\blambda}$.
We extract a numerical condition which is equivalent to this dimension condition.

\begin{lemma}
\label{Lem:MainLemma}
 Let $\blambda$ be a Schubert problem on $\Gr(k,n)$ with at least two solutions, 
 $d$ an integer between $0$ and $k{-}1$, 
 and $\bS=(S^1,\dotsc,S^s)$ with $S^i\subset[m]$ and $|S^i|\leq d$.
 Then $\dim \calC_{\bS, d} = \dim \calY_{\blambda}$ if and only if 
\begin{equation}
\label{Eq:Excess}
  \sum_{i=1}^s \sum_{j \in S^i} (\lambda^i_j - k + d + |[j] \smallsetminus S^i|)
    \  =\  d(n-2k+d) \,.
\end{equation}
\end{lemma}

\begin{proof}
 By Lemma~\ref{Lem:Components},  $C_{\bS, d}$ has dimension 
 \begin{multline*}
  \qquad
   \sum_{i=1}^s \bigl(N(\lambda^i) - 2|\lambda^i| 
   \ +\  \sum_{j \in S^i} (\lambda^i_j - k + d + |[j] \smallsetminus S^i|)\bigr)\\
   \ =\ N(\blambda) - 2k(n-k)\ +\ 
   \sum_{i=1}^s\sum_{j \in S^i} (\lambda^i_j - k + d + |[j] \smallsetminus S^i|)\,.
  \qquad
 \end{multline*}
Since $\dim \calC_{\bS, d} = \dim C_{\bS, d} + \dim \calO_d$, and 
$\dim \calO_d = 2k(n{-}k) - d(n{-}2k{+}d)$, we have 
\[
   \dim \calC_{\bS, d}\ = \
    N(\blambda)\ -\ d(n-2k+d)\ +\ 
   \sum_{i=1}^s\sum_{j \in S^i} (\lambda^i_j - k + d + |[j] \smallsetminus S^i|)\,.
\]
Since $\dim \calY_{\blambda} = N(\blambda)$, the result follows.
\end{proof}

%
\section{Double transitivity}
\label{S:MR}

We use Lemma~\ref{Lem:MainLemma} to show that if a Schubert problem
$\blambda$ on $\Gr(k,n)$ is either special or if $k\leq 3$, then 
its Galois group $\calG(\blambda)$ is doubly transitive.

%
\subsection{Special Schubert problems}

A \demph{special Schubert problem} is a Schubert problem all of whose conditions are
special, that is, a list $\blambda = (\lambda^i, \ldots, \lambda^s)$ in which each
partition consists of a single row.
Writing $\lambda^i = (a_i, 0, \dotsc, 0)$ for $i=1,\dotsc,s$, the 
non-negative integers $a_i$ satisfy the numerical condition
\[
    \sum_{i=1}^s a_i\ =\ k(n-k)\,.
\]
We identify $\blambda$ with the sequence $(a_1, \dotsc, a_s)$. Thus, $a_i$ is 
always understood to be $\lambda_1^i$.

\begin{theorem}\label{Th:Gspec_double}
 Every special Schubert problem has doubly transitive Galois group.
\end{theorem}

\begin{proof}
 Let $a$ be a special Schubert condition on $\Gr(k,n)$ and $H_1,H_2\in\Gr(k,n)$ with 
 $H_1\neq H_2$.  
 By Lemma~\ref{Lem:SpecComponents}, the intersection $\Psi_a H_1\cap \Psi_a H_2$ has possibly two 
 components $U_0(a,H_1,H_2)$ and $U_1(a,H_1,H_2)$, which are the closures of the sets
 $U_{\emptyset}(a,H_1,H_2)$ and $U_{\{1\}}(a,H_1,H_2)$ of Lemma~\ref{Lem:Components}.
 We adapt the notation of Proposition~\ref{P:XtwoComponents} to that of 
 Lemma~\ref{Lem:SpecComponents}. 
 The fiber $\pi^{-1}(H_1,H_2)$ has a decomposition into sets of the form
\[ 
    \defcolor{C_{T,d}(H_1,H_2)}\ :=\ 
    \prod_{i \not\in T} U_0(a_i,H_1,H_2) \times \prod_{i \in T} U_1(a_i,H_1,H_2) \,,
\]
 where $d=\dim H_1\cap H_2$ and $T\subset[s]$.
 This is a mild change in notation from Lemma~\ref{Lem:MainLemma},
 where $C_{\mathbf{T},d}$ was indexed by a sequence 
 $(T_1, \ldots, T_s)$,  with $T_i = \emptyset$ when $i \not\in T$, and $T_i = \{1 \}$
 otherwise.  

 Let $\calC_{T,d}$ be the subset of $\pi^{-1}(\calO_d)$ whose fiber over $(H_1,H_2)$ is 
 $C_{T,d}(H_1,H_2)$.
 Then the condition~\eqref{Eq:Excess} that $\dim\calC_{T,d}=N(\blambda)=\dim\calY_{\blambda}$
 becomes  
 \begin{equation}
  \label{Eq:SpecialExcess}
   \sum_{i \in T} (a_i - k + d)\ =\ d(n-2k+d)  \,.
 \end{equation}
 
 Let $b:= \max\{0, 2k{-}n\}$ so that $\calO_b$ is Zariski dense in $\Gr(k,n)^2$.
 We show that the only set $\calC_{T,d}$ which maps 
 dominantly onto $\calY_{\blambda}$ is the unique component of $\pi^{-1}(\calO_b)$.
 We do this by first showing that $\pi^{-1}(\calO_b)$ has dimension $N(\blambda)$ and then that
 no component of $\pi^{-1}(\calO_d)$ for $d>b$ which has dimension $N(\blambda)$ maps
 dominantly to $\calY_{\blambda}$.
 This implies that $\calX^{(2)}_{\blambda}$ is irreducible and thus $\calG(\blambda)$ is
 doubly transitive.   

 First suppose that $(H_1,H_2)\in\calO_b$.  
 By Lemma~\ref{Lem:SpecComponents}, the $i$th factor $\Psi_{a_i}H_1\cap \Psi_{a_i}H_2$
 in~\eqref{Eq:2-fiber} has a unique component which has dimension $N(a_i){-}2a_i$. 
 Indeed, if $b=0$ so that $(H_1,H_2)\in\calO_0$, then by Lemma~\ref{Lem:SpecComponents}, 
 $\Psi_{a_i}H_1\cap \Psi_{a_i}H_2$ is irreducible of dimension $N(a_i){-}2a_i$.
 If $b>0$, then $k{-}b=n{-}k$.
 Since $\blambda$ is a reduced special Schubert problem, we have $a_i<n{-}k$ and so
 $a_i{+}b<k$ and again Lemma~\ref{Lem:SpecComponents} implies that 
 $\Psi_{a_i}H_1\cap \Psi_{a_i}H_2$ is irreducible of dimension $N(a_i){-}2a_i$.
  Hence $\pi^{-1}(H_1,H_2)$ is irreducible of dimension 
\[
    N(\blambda)-2\sum_{i=1}^s a_i\ =\ N(\blambda) - 2k(n-k)\,.
\]
Since $\dim \calO_b = 2k(n-k)$, we have that $\dim \pi^{-1}(\calO_b) = N(\blambda)$.

We now show that if $d>b$, then no set of the form $\calC_{T,d}$ of dimension $N(\blambda)$
maps dominantly to $\calY_{\blambda}$.
Suppose that $d>b$ and $T\subset[s]$ satisfies~\eqref{Eq:SpecialExcess} so that 
$\dim\calC_{T,d}=N(\blambda)$.
We claim that the image $f(\mathcal{C}_{T,d})$ of $\calC_{T,d}$ in $\calY_{\blambda}$ is
not dense. 

Suppose not.
Then $f(\calC_{T,d})$ meets every open subset of $\calY_{\blambda}$.
For $(H_1,H_2)\in\calO_d$, if $(K_1,\dotsc,K_s)\in C_{T,d}(H_1,H_2)$, 
then the definition of $C_{T,d}(H_1,H_2)$ (via $U_1(a_i,H_1,H_2)$ for $i\in T$), implies that 
if $i\in T$, then $K_i$ meets the $d$-plane $L:=H_1\cap H_2$ nontrivially.
Since $\dim K_i=n{-}k{+}1{-}a_i=n{-}d{+}1{-}(a_i{+}k{-}d)$, we see that 
$L$ lies in the intersection of Schubert varieties in $\Gr(d,n)$,
 \begin{equation}\label{Eq:special_int}
  \bigcap_{i\in T} \Omega_{a_i{+}k{-}d} K_i\,,
 \end{equation}
and so this intersection is nonempty.
Since $f(\calC_{T,d})$ meets every open subset of $\calY_{\blambda}$, it 
meets the subset consisting of $(K_1,\dotsc,K_s)$ which are in sufficiently general
position in that the intersection~\eqref{Eq:special_int} has the expected codimension
$\sum_{i\in T}(a_i{+}k{-}d)$.

Suppose that $(H_1,H_2)$ is a point of $\calO_d$ with 
$(K_1,\dotsc,K_s)\in C_{T,d}(H_1,H_2)$, which is in this general position.
Since~\eqref{Eq:special_int} is nonempty, its codimension is at most the dimension of
$\Gr(d,n)$.
That is, 
\[
     d(n-d)\ \geq\    \sum_{i\in T} (a_i+k-d)\,.
\]
Rewrite the sum as $2(k-d)|T| + \sum_{i\in T}(a_i-k+d)$
and then use ~\eqref{Eq:SpecialExcess} to obtain
 \[
   d(n-d)\ \geq\  2(k-d)|T| + d(n-2k+d)\ =\ 
   d(n-d) - 2(d-|T|)(k-d)\,.
 \]
 Thus $|T| \leq d$. 
 If we set $A := \max_{i} \{ a_i \}$, then we have
\[
   d(A-k+d)\ \geq\ 
   \sum_{i\in T}(a_i-k+d)\ 
   \geq\  d(n-2k+d)\,
\]
whence $A \geq n-k$. 
Thus there is  some $i\in T$ for which $a_i \geq n{-}k$, contradicting our assumption that
every $a_i < n-k$ and $\blambda$ is reduced. 
\end{proof}

%
\subsection{Double transitivity for Galois groups of Schubert problems in $\Gr(3,n)$}

\begin{theorem}
\label{Th:Gr3n_double}
  Every Schubert problem in $\Gr(3,n)$ has a doubly transitive Galois group.
\end{theorem}

Let $\blambda$ be a Schubert problem in $\Gr(3,n)$.
We may assume that $\blambda$ is reduced.
If not, then  by Proposition~\ref{Prop:reduced} it is equivalent to a
reduced Schubert problem on either a $\Gr(3,m)$, a $\Gr(2,m)$, or a $\Gr(1,m)$ with
$m<n$. 
All reduced Schubert problems on $\Gr(1,m)=\P^{m-1}$ have a single solution.
Theorem~\ref{Th:Gr3n_double} covers the first case of $\Gr(3,m)$ and
Theorem~\ref{Th:Gspec_double} covers the second case of $\Gr(2,m)$, for
every reduced Schubert problem on $\Gr(2,m)$ is special~\cite[\S~1.3]{BdCS}.

Since ${\blambda}$ is reduced, for every two partitions $\mu,\nu\in\blambda$, we have 
\[
   \mu_1\ <\ n{-}3\,,\quad
   \mu_3\ =\ 0\,,\quad
   \mu_2+\nu_2\ <\ n{-}3\,,\quad \mbox{ and }\quad
   \mu_1+\nu_2\ \leq\ n{-}3\,.
\]
In particular, these imply that at most one partition $\mu$ in $\blambda$ has
$\mu_2=n{-}4$.

Indices $\bS$ of sets $\calC_{\bS,d}$ in the decomposition of $\pi^{-1}(\calO_d)$ are
lists $\bS=(S^1,\dotsc,S^s)$ of subsets $S^i\subset\{1,2\}$ where $S^i\subset\{1\}$ if
$\lambda^i_2=0$ and $|S^i|\leq d$.
For $I \subseteq \{1,2 \}$ define $S_I := \{i \in [s]\mid S^i = I \}$. 
Then we can rewrite~\eqref{Eq:Excess} as 
 \begin{equation}
 \label{Eq:3-Excess}
  d(n-6+d)  \  =\ 
   \sum_{i \in S_{\{1\}}} (\lambda_1^i - 3 + d) + \sum_{i \in S_{\{2\}}} (\lambda_2^i - 2 + d) 
       + \sum_{i \in S_{\{1,2\}}} (\lambda^i_1 + \lambda^i_2 - 6 + 2d)\,.
 \end{equation}

We prove Theorem~\ref{Th:Gr3n_double} by showing that there is at most one set
$\calC_{\bS,d}$ with $\bS$ and $d$ satisfying~\eqref{Eq:3-Excess},  (and thus having
dimension $N(\blambda)$) which maps dominantly to $\calY_{\blambda}$.
Since $X^{(2)}\neq\emptyset$, this set maps dominantly to $Y_{\blambda}$, and so it is
dense in $X^{(2)}$.
This implies that $\calX^{(2)}$ is irreducible, and thus $\calG(\blambda)$ is
doubly transitive.
We consider each case $d=0,1,2$ separately in the following three lemmas.

The set $\calY_{\blambda}$ consists of $s$-tuples of partial flags
 $\calF=(E^1\subset F^1,\dotsc,E^s\subset F^s)$ where, if the $i$th partition is special,
 $\lambda^i_2=0$ and so $m_i=1$, then $F^i$ is omitted as it is not needed to define the
 Schubert variety indexed by $\lambda^i$.
However, to simplify the arguments that follow, we will write these flags as if they are
all two-step flags.
The resulting ambiguity may be remedied by setting $F^i=\C^n$ when $\lambda^i_2=0$.

\begin{lemma}
\label{Lem:Gr3n_part1}
 If there is a partition $\mu$ in $\blambda$ with $\mu_2=n{-}4$, then 
 $\pi^{-1}(\calO_0)=\emptyset$.
 In all other cases, $\pi^{-1}(\calO_0)$ is irreducible and has dimension $N(\blambda)$.
\end{lemma}

\begin{proof}
 First suppose that there is a partition $\mu$ in $\blambda$ with $\mu_2=n{-}4$.
 If $H\in\Omega_\mu(E\subset F)$, then $\dim F=3$ and $\dim H\cap F\geq 2$.
 In particular this means that if $H_1,H_2\in\Omega_\mu(E\subset F)$, then 
 $\dim H_1\cap H_2\geq 1$ and so $(H_1,H_2)\not\in\calO_0$.
 Thus  $\pi^{-1}(\calO_0)=\emptyset$.

 Suppose now that every partition $\mu$ in $\blambda$ has $\mu_2<n{-}4$.
 The only index $\bS$ for $d=0$ is when $\bS=(\emptyset,\dotsc,\emptyset)$.
 Then $(\bS,0)$ satisfies~\eqref{Eq:3-Excess} and so $\calC_{\bS,0}$ has dimension  
 $N(\blambda)$.
 As the index $\bS$ is unique, 
 $\pi^{-1}(\calO_0)$ is irreducible and has dimension $N(\blambda)$.
\end{proof}

\begin{lemma}
\label{Lem:Gr3n_part2}
 There is a set $\calC_{\bS,1}$ of dimension $N(\blambda)$ if and only if
 there is a partition $\mu$ in $\blambda$ with $\mu_2=n{-}4$, and in that 
 case the index $\bS$ is unique.
\end{lemma}

\begin{proof}
 First note that if there is a partition $\mu$ in $\blambda$ with $\mu_2=n{-}4$, then it
 is unique, say $\mu=\lambda^1$.
 Consider the index $\bS=(S^1,\dotsc,S^s)$, where $S^1=\{2\}$, and $S^j = \emptyset$ for 
 all $j>1$. 
 Then $\bS$ satisfies~\eqref{Eq:3-Excess} with $d=1$, and so 
 $\calC_{\bS,1}$ has dimension $N(\blambda)$.

 Conversely, suppose that the index $\bS$ satisfies~\eqref{Eq:3-Excess} for $d=1$.
 One of the sets $S_{\{1\}}$, $S_{\{2\}}$, or $S_{\{1,2\}}$ must be nonempty.
 Since $d=1$, $S_{\{1,2\}} = \emptyset$, because $|S^i|\leq d=1$.

 Let $(H_1,H_2)\in\calO_1$ and set $\defcolor{L}:= H_1\cap H_2$, which has dimension 1.
 Suppose that $(E^1 \subset F^1,\dotsc,E^s \subset F^s)\in C_{\bS}(H_1,H_2)$ and that the
 flags are in general position.
 Then by the definition of $U_{\{1\}}$ and $U_{\{2\}}$, we have 
\[
    L\ \subset\ \bigcap_{i \in S_{\{1\}}} E^i\ \cap\ \bigcap_{i \in S_{\{2\}}} F^i\,.
\]
  Thus $\codim L \geq \sum_{i\in S_{\{1\}}}\codim E^i + \sum_{i \in S_{\{2\}}} \codim
  F^i$ and so  
 \begin{equation}\label{Eq:line_codim}
   n{-}1\ \geq\ 
   \sum_{i \in S_{\{1\}}} (\lambda^i_1 + 2)\ +\ \sum_{i\in S_{\{2\}}} (\lambda^i_2+1)\,.
 \end{equation}
Subtracting~\eqref{Eq:3-Excess} from this and dividing by 2 gives 
$2 \geq 2|S_{\{1\}}| + |S_{\{2\}}|$.
If $S_{\{1\}}=\{i\}$ then $S_{\{2\}}=\emptyset$ and~\eqref{Eq:3-Excess} implies that
$\lambda^i_1 = n{-}3$, and thus $\blambda$ is not a reduced Schubert problem, which is a
contradiction.  
Thus $S_{\{1\}}=\emptyset$ and we have $|S_{\{2\}}|\leq 2$.

If $S_{\{2\}}=\{i,j\}$ with $i\neq j$, then~\eqref{Eq:3-Excess} implies that 
$\lambda^i_2+\lambda^j_2=n-3$, which implies that $\blambda$ is not reduced, a
contradiction.
We are left with the case of $S_{\{2\}} = \{i \}$. 
Then~\eqref{Eq:3-Excess} implies that $\lambda^i_2 = n{-}4$,
which completes the proof.
\end{proof}

\begin{lemma}
\label{Lem:Gr3n_part3}
 No set $\calC_{\bS,2}$ of dimension $N(\blambda)$ maps dominantly to $\calY_{\blambda}$.
\end{lemma}

We first state an auxiliary lemma which will be used in the proof of
Lemma~\ref{Lem:Gr3n_part3}. 

\begin{lemma}\label{Lem:codim}
  Let $\lambda$ be a Schubert condition for $\Gr(3,n)$, $(H_1,H_2)\in\calO_2$, and  
  $(E\subset F)\in \Psi^\circ_\lambda H_1\cap \Psi^\circ_\lambda H_2$.
  Setting $L:=H_1\cap H_2\in\Gr(2,n)$ and $M:=\langle H_1,H_2\rangle\in\Gr(4,n)$, 
  then $L$ and $M$ satisfy Schubert conditions $\mu$ and $\nu$, respectively, with respect
  to the flag $E\subset F$ 
  depending upon which subset 
  $U_S(\lambda,H_1,H_2)$ they belong to according to the
  following table.
\[
\begin{tabular}{|r||c|c|c|c|}\hline
   $S$ &$\emptyset$&$\{1\}$&$\{2\}$&$\{1,2\}$\rule{0pt}{11pt}\\\hline\hline\rule{0pt}{11pt}
 $\mu$&$(0,0)$&$(\lambda_1{+}1,0)$&$ (\lambda_2,0) $&$ (\lambda_1{+}1,\lambda_2{+}1)$\\\hline
 $\nu$&$(a,a,\lambda_2{+}1,\lambda_2{+}1)$&$ (a{-}1,\lambda_2,\lambda_2,0)$&$\rule{0pt}{11pt}
     (\lambda_1,\lambda_1,\lambda_2,0)$&$(\lambda_1{-}1,\lambda_2{-}1,0,0)$\\\hline
\end{tabular}
\]
 Here, $a=\max\{\lambda_1,\lambda_2+1\}$.
\end{lemma}

\begin{proof}[Proof of Lemma~\ref{Lem:Gr3n_part3}]
  Suppose that the index $\bS$ satisfies~\eqref{Eq:3-Excess} for $d=2$. 
  We will show that the projection of $\calC_{\bS,2}$ to $\calY_\blambda$ is not dense.

  Let $(H_1,H_2)\in\calO_2$ and suppose that
  $\calF=(E_1 \subset F_1,\dotsc,E_s \subset F_s)\in C_{\bS,2}(H_1,H_2)$.
  Setting $L := H_1 \cap H_2\in\Gr(2,n)$ and $M := \langle H_1, H_2 \rangle\in\Gr(4,n)$
  then these lie in intersections of Schubert varieties given by flags in
  $\mathcal{F}$ as detailed in Lemma~\ref{Lem:codim}. 
  If we assume that the flags in $\calF$ are in general position so that those Schubert
  varieties intersect in the expected dimensions, then 
  we show that the nonemptiness of
  those intersections gives an inconsistent system of linear equations and inequalities,
  which proves the lemma.

  The condition that the intersection involving $L$ is nonempty is
 \begin{equation}\label{Eq:LNonEmpty}
   2(n-2)\ \geq\ \sum_{i \in S_{\{1\}}} (\lambda^i_1+1) 
       + \sum_{i \in S_{\{2\}}} \lambda^i_2 
       + \sum_{i \in S_{\{1,2\}}} (\lambda^i_1 + \lambda^i_2 + 2)\,.
 \end{equation}
  Subtracting~\eqref{Eq:3-Excess} from this and dividing by 2 gives
 \begin{equation}\label{Eq:sets_ineqs}
  2\ \geq\ |S_{\{1\}}| + 2|S_{\{1,2\}}|\,.
 \end{equation}
 The condition that the intersection involving $M$ is nonempty is
 \begin{multline}\label{Eq:MNonEmpty}
   \qquad 4(n-4)\ \geq\ 
       2\sum_{i\in S_{\emptyset}}(\lambda^i_1 + \lambda^i_2+1)
       +  \sum_{i \in S_{\{1\}}} (\lambda^i_1 + 2\lambda^i_2-1) \\
       + \sum_{i \in S_{\{2\}}} (2\lambda^i_1 + \lambda^i_2) 
       + \sum_{i \in S_{\{1,2\}}}(\lambda^i_1 + \lambda^i_2 - 2)\,.\qquad
 \end{multline}
  Adding~\eqref{Eq:3-Excess} to this and dividing by two gives
 \[
   3n-12\ \geq\ \sum_{i=1}^s|\lambda^i| +|S_{\emptyset}|- |S_{\{1\}}| -2 |S_{\{1,2\}}|
 \]
 Since $\blambda$ is a Schubert problem on $\Gr(3,n)$, $\sum_{i=1}^s |\lambda^i| = 3(n-3)$ and
 so we obtain $|S_1| + 2|S_{\{1,2\}}| \geq 3+|S_{\emptyset}|$, which
 contradicts~\eqref{Eq:sets_ineqs}. 
 Thus the projection of the set $\calC_{\bS,2}$ to $\calY_{\blambda}$ cannot meet the
 subset consisting of $s$-tuples of flags that are in general position for these two
 intersections of Schubert varieties.
\end{proof}

\begin{proof}[Proof of Lemma~\ref{Lem:codim}]
 We determine the dimensions of $L\cap E$, $L\cap F$, $M\cap E$, and
 $M\cap F$ and then apply the definition~\eqref{Eq:GrassSchubVar} to obtain the
 partitions encoding those conditions.
 Recall that $\dim E = n-2-\lambda_1$ and $\dim F=n-1-\lambda_2$, where
 $\lambda=(\lambda_1,\lambda_2,0)$.
 We use the notation of the proof of Lemma~\ref{Lem:Components}, for $i=1,2$, setting
 $h^i_1:= H_i\cap E$ and $h^i_2:=H_i\cap F$.
 Then the conditions on $H_i$ are that $\dim h^i_1 = 1$ and $\dim h^i_2=2$, for each
 $i=1,2$. 
 If we set $\defcolor{\Lambda_j}:= \langle h^1_j, h^2_j \rangle$ for $j=1,2$, then 
 $\dim \Lambda_1=2-|[1]\cap S|$ and similarly 
 $\dim \Lambda_2=4-|[2]\cap S|$.
 These are lower bounds on the dimension of $M\cap E$ and $M\cap F$.

 Suppose first that $(E\subset F)\in U_\emptyset$ so that $L\cap E=L\cap F=\{0\}$, and 
 so $\dim M\cap E\geq 2$ and $\dim M\cap F=4$.
 Then the flags  $E\subset F$ impose no conditions on $L$ so that $\mu=(0,0)$.
 If $\dim E < \dim F-1$, so that $\lambda_1>\lambda_2$, then 
 $\dim M\cap E= 2$ and $\nu$ is $(\lambda_1,\lambda_1,\lambda_2+1,\lambda_2+1)$.
 However, if $\lambda_1=\lambda_2$ so that $\dim E = \dim F-1$ then
 $\dim M\cap E= 3$, and we have 
 $\nu=(\lambda_2+1,\lambda_2+1,\lambda_2+1,\lambda_2+1)$.

 Suppose that $(E\subset F)\in U_{\{1\}}$ so that $L\cap E=L\cap F$ and this has
 dimension 1, so that $\mu=(\lambda_1+1,0)$.
 Then $\dim M\cap F=3$ and $\dim M\cap E\geq 1$.
 As in the previous case, the dimension of $M\cap E$ depends upon whether or not
 $\lambda_1=\lambda_2$, and so we get 
 $\nu=(\max\{\lambda_1{-}1,\lambda_2\},\lambda_2,\lambda_2,0)$.

 Suppose that $(E\subset F)\in U_{\{2\}}$ so that $L\cap E=\{0\}$, but
 $\dim L\cap F=1$.
 Then $L\in\Omega_{(\lambda_2,0)}F$.
 We have that $\dim L\cap E = 2$ and $\dim L\cap M=3$, so that 
 $\nu=(\lambda_1,\lambda_1,\lambda_2,0)$.

 Finally, suppose that $(E\subset F)\in U_{\{1,2\}}$ so that $\dim L\cap E=1$ and 
 $\dim L\cap F=2$.
 Then $\mu=(\lambda_2+1,\lambda_2+1)$ and $\nu=(\lambda_1-1,\lambda_2-1,0,0)$.
\end{proof}

%

\section{Schubert problems on $\Gr(2,n)$ have at least alternating Galois groups}
\label{S:G2n}

The main result of~\cite{BdCS} is that every Schubert problem on
$\Gr(2,n)$ has at least alternating Galois group.
The proof relied on Vakil's Criterion (b) and used elementary, but very involved, estimates
of trigonometric integrals to establish an inequality between two Kostka numbers.
We use the stronger and simpler Vakil's Criterion (c) (which applies, by
Theorem~\ref{Th:Gspec_double}) to give a much simpler and purely combinatorial proof
of that result.

\begin{theorem}
\label{thm:G2nalternating}
 Every Schubert problem on $\Gr(2,n)$ has at least alternating Galois group.
\end{theorem}

It suffices to prove this for reduced Schubert problems, by Proposition~\ref{Prop:reduced}. 
By Definition~\ref{D:reduced}, a reduced Schubert problem on $\Gr(2,n)$ has the
form $(a_1,\dotsc,a_s)$ with $\sum_i a_i=2(n{-}2)$ and 
$a_i+a_j\leq n{-}2$ for each $i,j$, in particular, a reduced Schubert problem is special.
Let $K(a_1,\dotsc,a_s)$ be the number of solutions to a special Schubert problem
$(a_1,\dotsc,a_s)$ on $\Gr(2,n)$.
This is a Kostka number, which counts the number of Young tableaux of shape $(n{-}2,n{-}2)$
and content $(a_1,\dotsc,a_s)$.
These are fillings of the Young diagram of shape
$(n{-}2,n{-}2)$ with $a_1$ ones, $a_2$ twos, through $a_s$ s's, such that the filling is
weakly increasing along the rows and strictly increasing 
down the columns. 

Section 2.1 of~\cite{BdCS} shows that there is a fiber diagram~\eqref{Eq:fiber_diagram},
such that $\calG(W_1 \to Z)$ and $\calG(W_2 \to Z)$ are isomorphic to the Galois groups of
the Schubert problems $(a_1, \ldots, a_{s-1}+a_s)$ and 
$(a_1, \ldots, a_{s-2}, a_{s-1}{-}1, a_s{-}1)$, respectively. 
This gives \demph{Schubert's recursion},
 \begin{equation}\label{Eq:SchubertRecursion}
   K(a_1,\ldots, a_s)\ =\
    K(a_1, \ldots, a_{s-2}, a_{s-1}{+}a_s)\ +\
   K(a_1, \ldots, a_{s-2}, a_{s-1}{-}1, a_s{-}1)\,.
 \end{equation}

The following proposition is proven in~\cite{BdCS} using a combinatorial injection.

\begin{proposition}[Lemma~11~\cite{BdCS}]
 \label{prop:injection}
  Let $(a_1, \ldots, a_s)$ be a reduced Schubert problem on $\Gr(2,n)$ such that 
  $a_{s-2}\leq a_{s-1}\leq a_s$ with $a_{s-2}<a_s$.
   Then 
 \[ 
    K(a_1,\ldots, a_{s-2}, a_{s-1}+a_s)\ <\
    K(a_1, \ldots, a_{s-3},a_s, a_{s-2}+a_{s-1})\,.
 \] 
\end{proposition}

We compute some Kostka numbers.

\begin{lemma}
\label{lem:multparts}
 Let $a, s$ be positive integers with $as = 2(n-2)$.
 Then
\begin{enumerate}
\item $K(a^2, 2a) = K(a^3) = 1$,
\item $K(a^4) = a+1$,
\item $K(a^3, 2a) = b + 1$ and $K(a^3, a{-}1,a{-}1) = \frac{5b^2+3b}{2}$, where $a = 2b$
        is even, 
\item $K(a^4, 2a) = \binom{a+2}{2}$,
\item $K(a^{s-2}, 2a) \geq (a+1)K(a^{s-4}, 2a) + K(a^{s-4})$ when $s \geq 7$.
\end{enumerate}
\end{lemma}

\begin{proof}
 Statements (1)---(3) are from Lemmas 8 and 9 of~\cite{BdCS}. 
 Statement (4) is simple, as $\binom{a+2}{2}$ is the number of triples $(x,y,z)$ of
 nonnegative integers whose sum is $a$.
 A Young tableau of shape $(n{-}2,n{-}2)$ and content $(a^4,2a)$ is determined by its
 second row which consists of $2a$ fives and $a$ remaining numbers which are some twos,
 threes, and fours.

 For (5), we construct tableaux of shape $(n{-}2,n{-}2)$ and content $(a^{s-2},2a)$
 from tableaux of shape $(n{-}2{-}a,n{-}2{-}a)$ and content $(a^{s-4},2a)$ and of
 shape $(n{-}2{-}2a,n{-}2{-}2a)$ and content $(a^{s-4})$.
 Let $s \geq 7$ and consider a tableau $T$ of shape $(n{-}2{-}a,n{-}2{-}a)$ and content 
 $(a^{s-4}, 2a)$.
 For each $i=0,\dotsc,a$, construct a tableau $T_i$ of   $(n{-}2,n{-}2)$ and content
 $(a^{s-2},2a)$ as follows.
 The second row of $T$ ends in $2a$ entries of $s{-}3$.
 Remove these to get a tableau $T'$ of shape $(n{-}2{-}a,n{-}2{-}3a)$ and content $(a^{s-4})$,
 which we extend to get $T_i$ as follows.
 Place $i$ entries of $s{-}3$ and $a{-}i$ entries of $s{-}2$ at the end of the first row of
 $T'$ and then $a{-}i$ entries of $s{-}3$ and $i$ entries of $s{-}2$ in the second row,
 followed by $2a$ entries of $s{-}1$ to obtain $T_i$.
\[
  T\ =\ 
  \raisebox{-13pt}{\begin{picture}(145,33)(5,0)
   \put(5,0){\includegraphics{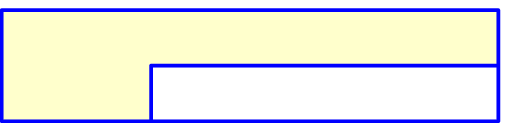}}
     \put(22.5,13){$T'$}
     \put(89,5){$s{-}3$}
  \end{picture}}
   \ \longmapsto\ 
  \raisebox{-22.5pt}{\begin{picture}(195,54)(5,-11)
   \put(5,0){\includegraphics{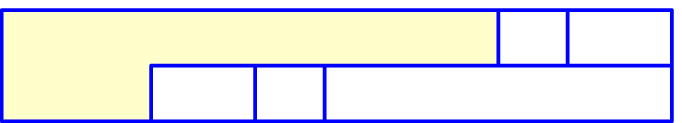}}
     \put(22.5,13){$T'$}
     \put(52,5){$s{-}3$}
     \put(139,5){$s{-}1$}
     \put(173,21){$s{-}2$}
     \put(156,36){$i$} \put(173,36){$a{-}i$}
     \put(85,-10){$i$} \put(52,-10){$a{-}i$} \put(143,-10){$2a$}
  \end{picture}}
  \ =:\ T_i
\]
 This gives $(a{+}1)K(a^{s-4},2a)$ tableaux of shape $(n{-}2,n{-}2)$ and content
 $(a^{s-2},2a)$.

 Given a tableau $T$ of shape $(n{-}2{-}2a,n{-}2{-}2a)$ and content $(a^{s-4})$, add
 $a$ entries of $s{-}3$ and $a$ of $s{-}2$ to the first row and $2a$ entries of
 $s{-}1$ to get a tableau of shape $(n{-}2,n{-}2)$ and content  $(a^{s-2},2a)$.
\[
  \raisebox{-13pt}{\begin{picture}(85,33)(5,0)
   \put(5,0){\includegraphics{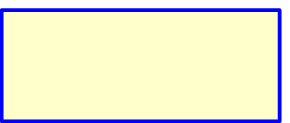}}
     \put(40,13){$T$}
  \end{picture}}
   \ \longmapsto\ 
  \raisebox{-13pt}{\begin{picture}(145,33)(0,0)
   \put(5,0){\includegraphics{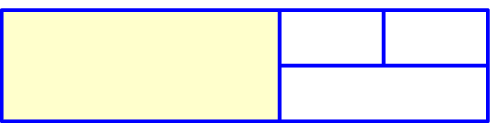}}
     \put(40,13){$T$}
     \put(90,21){$s{-}3$}
     \put(120,21){$s{-}2$}
     \put(105,5){$s{-}1$}
  \end{picture}}
\]
 This gives $K(a^{s-4})$ tableaux of shape $(n{-}2,n{-}2)$ and content
 $(a^{s-2},2a)$, all of which are different from the others we constructed, which proves
 statement~(5).
\end{proof}

\begin{proof}[Proof of Theorem~\ref{thm:G2nalternating}]
 We prove the result by induction on $s$ and $n$. 
 Let $\blambda=(a_1, \ldots, a_s)$ be a reduced Schubert problem.
 We may assume that $s\geq 4$, for otherwise $K(a_1,\dotsc,a_s)\leq 1$ and there is 
 nothing to prove.
 Since $\calG(\blambda)$ is doubly transitive by Theorem~\ref{Th:Gspec_double},
 Vakil's Criterion (c) and Schubert's recursion~\eqref{Eq:SchubertRecursion} imply that 
 $\calG(\blambda)$ is at least alternating if one of $K(a_1, \ldots, a_{s-1}+a_s)$ or 
 $K(a_1, \ldots,a_{s-1}{-}1,a_s{-}1)$ is not six, for some reordering of
 the list $a_1,\dotsc,a_s$.  

 Suppose first that not all the $a_i$ are equal and that they are ordered so that 
 $a_{s-2}\leq a_{s-1}\leq a_s$ with $a_{s-2}<a_s$.
 By~\eqref{Eq:SchubertRecursion}, $K(a_1,\dotsc,a_s)$ is equal to
 either the sum,
\[
   K(a_1,\dotsc,a_{s-2}, a_{s-1}+a_s)\ +\ 
   K(a_1,\dotsc,a_{s-2}, a_{s-1}{-}1, a_s{-}1)\,,
\]
 or the sum 
\[
   K(a_1,\dotsc,a_{s-3}, a_s, a_{s-1}+a_{s-2})\ +\ 
   K(a_1,\dotsc,a_{s-3}, a_s, a_{s-1}{-}1, a_{s-2}{-}1)\,.
\]
 By Proposition~\ref{prop:injection},
 $K(a_1,\dotsc,a_{s-2}, a_{s-1}+a_s)<K(a_1,\dotsc,a_{s-3}, a_s, a_{s-1}+a_{s-2})$, so they
 cannot both be six, which implies $\calG(\blambda)$ is  at least alternating.

 Assume now $a_1 = a_2 = \cdots = a_s = a$ for some positive integer $a$. 
 There are only two values of $(a,s)$ for which  $K(a^{s-2}, 2a) = 6$, and 
 for both of these $K(a^{s-2},a{-}1,a{-}1)\neq 6$, which implies that 
 $\calG(a^s)$ is at least alternating.

 Indeed, by Lemma~\ref{lem:multparts}, if $s=4$, then $K(a^2,2a)=1$.
 If $s=5$, then $K(a^3,2a)=6$ only for $a=10$, and then 
 $K(10^3,9,9) = \frac{5\cdot5^2 + 3\cdot 5}{2} = 70 > 6$.
 If $s=6$, then $K(a^4,2a)=6$ only for $a=2$, and then 
 $K(2^4,1,1) = 9 > 6$.

 We show that if $s>6$, then $K(a^{s-2}, 2a)>6$.
 By Lemma~\ref{lem:multparts} (5), 
\[
   K(a^{s-2}, 2a)\ \geq\ (a+1)K(a^{s-4},2a)\ +\ K(a^{s-4})\ >\ K(a^{s-4},2a)\,.
\]
 So it suffices to show $K(a^{s-2}, 2a)>6$ for $s=7,8$.
 When $s=8$, this becomes
\[
   K(a^{6}, 2a)\ \geq\ (a+1)K(a^{4},2a)\ +\ K(a^{4})\ =\ 
    (a+1) \binom{a+2}{2} + a + 1 \ \geq\ 8\,,
\]
  and when $s=7$, $a=2b$ is even, and we have 
\[
    K(a^{5}, 2a)\ \geq\
    (a+1)K(a^3,2a)\ +\ K(a^3)\ \geq\
    (a+1) \left(b + 1\right) + 1\ \geq\ 7\,.
\]
 This completes the proof.
\end{proof}

%
\section{Galois groups of Schubert problems on $\Gr(4,8)$}
\label{S:IM}

Using Vakil's Criteria and the Frobenius Algorithm of~\cite[\S~5.4]{MSJ} (a symbolic
method to prove that a Galois group is full symmetric by computing cycle types of
elements), we study the Galois group of every reduced
Schubert problem on $\Gr(4,8)$.
A nonreduced problem in $\Gr(4,8)$ is equivalent to a problem in some $\Gr(k,n)$ with
$k\leq 4$ and $n<8$, and Vakil earlier showed that these problems have at least
alternating Galois groups.
There are 2987 reduced Schubert problems on $\Gr(4,8)$ having two or more solutions.
%
%
All except fourteen have at least alternating Galois group with many 
known to be the full symmetric group. 
Each of these fourteen have imprimitive Galois group, which we determine, and 
they fall into three families according to their geometry.
Much of this computation, except for the determination of the Galois groups of the
  last thirteen Schubert problems, was done by Vakil in~\cite{Va06b}.

Vakil wrote a maple script based on his geometric Littlewood-Richardson rule and his
Criterion (b) to test which Schubert problems had Galois groups that are at least
alternating~\cite{Va06b}. 
We altered it to only test reduced Schubert problems on $\Gr(4,8)$ and found 28
problems for which it was inconclusive. 
We list them, using a compact product notation in which the Schubert
problem $(\sTT,\sTT,\sTT,\sI,\sI,\sI,\sI)$ with 32 solutions is written 
$\sTT^3\cdot\sI^4=32$.
 \begin{center}
  $\sI^{16}=24024$\,,\ \ 
  $\sTT\cdot\sI^{12}=2640$\,,\\\rule{0pt}{14pt}
  $\sIII\cdot\sTh\cdot\sI^{10}=420$\,,\ \ 
  $\sTT^2\cdot\sI^8=280$\,,\ \ 
  $\sTT\cdot\sTI^4=42$\,,\ \ 
  $\sThTT\cdot\sIII\cdot\sI^6=36$\,,\\ 
  $\sTT^3\cdot\sI^4=32$\,,\ \ 
  $\sThThTh\cdot\sI^5=20$\,,\ \ \ \rule{0pt}{14pt}
  $\sThI\cdot\sTI\cdot\sTh^3=6$\,,\  \ 
  $\sIII^4\cdot\sII\cdot\sT=6$\,,\\ 
  $\sTh^4\cdot\sII\cdot\sT=6$\,,\ \ \rule{0pt}{14pt}
  $\sThTI^2\cdot\sII\cdot\sT=6$\,,\ \ 
  $\sThTh\cdot\sII^4\cdot\sT=6$\,,\ \ 
  $\sTTT\cdot\sII\cdot\sT^4=6$\,,\\\rule{0pt}{16pt}
  $\sTT^4=6$\,,\\\rule{0pt}{16pt}
  $\sTT^2\cdot\sTII\cdot\sThI\ =\ 
   \sTT^2\cdot\sTII\cdot\sTh\cdot\sI\ =\ 
   \sTT^2\cdot\sIII\cdot\sThI\cdot\sI\ =\ 
   \sTT^2\cdot\sIII\cdot\sTh\cdot\sI^2\ =\ 4$\,,\\\rule{0pt}{16pt}
  $\sTII^2\cdot\sThI^2\ =\ 
   \sTII^2\cdot\sThI\cdot\sTh\cdot\sI\ =\ 
   \sTII\cdot\sIII\cdot\sThI^2\cdot\sI\ =\ 
   \sTII^2\cdot\sTh^2\cdot\sI^2\ = \ 
   \sTII\cdot\sIII\cdot\sThI\cdot\sTh\cdot\sI^2$\\
  $=\ \sIII^2\cdot\sThI^2\cdot\sI^2\ =\ 
     \sTII\cdot\sIII\cdot\sTh^2\cdot\sI^3\ =\ 
      \sIII^2\cdot\sThI\cdot\sTh\cdot\sI^3\ =\
   \sIII^2\cdot\sTh^2\cdot\sI^4\ =\ 4$\,.
\end{center}

In Subsection~\ref{S:largeSP} we show that the two Schubert problems in the first row have
at least alternating Galois groups.
We used the Frobenius Algorithm~\cite[\S~5.4]{MSJ} to show that the next twelve have full
symmetric Galois group.
The remaining fourteen on the last four lines have imprimitive Galois groups.
They are grouped by similar geometry, which we indicate by strings of equalities.
We describe one problem from each family in the remaining three subsections.

%
\subsection{Two large Schubert problems}\label{S:largeSP}

The Schubert problem $\blambda\colon\sI^{16}=24024$ is special and therefore has doubly
transitive Galois group by Theorem~\ref{Th:Gspec_double}.  
It asks for the 4-planes $H$ in $\C^8$ that meet 16 general 4-planes $K_1,\dotsc,K_{16}$
nontrivially. 
If $(K_1,K_2)\in\calO_3$ so that $\defcolor{m}:=K_1\cap K_2$ is a 3-plane and
$\defcolor{M}:=\langle K_1,K_2\rangle$ is a 5-plane, then Lemma~\ref{Lem:SpecComponents}
shows that 
 \begin{equation}
  \label{Eq:two_components}
   \Omega_{\sI}K_1\cap \Omega_{\sI}K_2\ =\ 
   \Omega_{\sT} m \cup \Omega_{\sII} M\,.
 \end{equation}
Thus if $\defcolor{\calZ}=\calO_3\times \Gr(4,8)^{14}\subset\calY_{\blambda}$ and
$\calW\to\calZ$ is the restriction of $\calX_{\blambda}$ to $\calZ$ (as
in~\eqref{Eq:fiber_diagram}), then $\calW= \calX_{\bmu}\cup\calX_{\bnu}$ where
$\calX_{\bmu}$ and $\calX_{\bnu}$ are (essentially) total spaces of the Schubert
problems 
\[
  \bmu\ \colon\ \sT\cdot \sI^{14}\ =\ 12012
   \qquad\mbox{and}\qquad
  \bnu\ \colon\ \sII\cdot \sI^{14}\ =\ 12012\,.
\]
These are equivalent dual problems and each was found to have at least alternating Galois
group by Vakil's maple script.
As the original problem $\sI^{16}=24024$ is doubly transitive, Vakil's Criterion (c) 
and Remark~\ref{R:Vakil} implies that its  Galois group is at least alternating.

The Schubert problem  $\blambda\colon\sTT\cdot \sI^{14}=2640$ also has doubly transitive
Galois group. 
To see this note that by Lemma~\ref{Lem:Components} the subset $U_S(\sTT,H_1,H_2)$ of 
$\Psi^\circ_{\ssTT}H_1\cap \Psi^\circ_{\ssTT}H_2$ has dimension at most $N(\sTT)=8$ unless 
$d=\dim H_1\cap H_2 = 2$ and $S=\{2\}$, and that $U_{\{2\}}$ has dimension 9.
Since $\Psi^\circ_{\sI}H_1\cap \Psi^\circ_{\sI}H_2$  has dimension $N(\sI)=14$ for $H_1\neq H_2$, we
see that the only set $\calC_{\bS,d}$ having dimension $N(\blambda)$ is when $d=0$ and each
component of $\bS$ is $\emptyset$.

The special position~\eqref{Eq:two_components} gives a subset
$\calZ\subset\calY_{\blambda}$ with the restriction of $\calX_{\blambda}$ to $\calZ$ having two
components, each essentially the total space of one of the Schubert problems
\[
   \sTT\cdot\sII\cdot\sI^{10}\ =\ 1320
  \qquad\mbox{and}\qquad
   \sTT\cdot\sT\cdot\sI^{10}\ =\ 1320\,.
\]
These equivalent dual Schubert problems were found to have at least alternating Galois
group by Vakil's maple script.
As the original problem $\sTT\cdot\sI^{12}=2640$ is doubly transitive, Vakil's Criterion
(c) implies that its  Galois group is at least alternating.

%
\subsection{The Schubert problem $\sTT^4=6$.}\label{S:TT^4}

Derksen discovered that this Schubert problem has Galois group $S_4$ and it was described by
Vakil~\cite{Va06b}.
An instance is given by four $4$-planes $K_1,\dotsc,K_4$ in general position in
$\C^8$. 
Its solutions are those $H\in\Gr(4,8)$ for which $\dim H\cap K_i\geq 2$ for 
$i=1,\dotsc,4$.

Consider the {\it auxiliary} problem $\sTh^4$ in $\Gr(2,8)$
given by $K_1,\dotsc,K_4$.
This asks for those $h\in\Gr(2,8)$ with $\dim h\cap K_i\geq 1$ for $i=1,\dotsc,4$.
There are four solutions $h_1,\dotsc,h_4$ to this problem, and its Galois group is 
the full symmetric group $S_4$.

Each of the 4-planes $\defcolor{H_{a,b}}:=h_a\oplus h_b$ 
will meet each $K_i$ in a 2-plane, and so they are solutions to the
original problem. 
In fact, they are the only solutions.
It follows that the Galois group of $\sTT^4=6$ is $S_4$ acting on the
pairs $\{h_a,h_b\}$.
This is an imprimitive permutation group as it preserves the partition
\[
   \{H_{12},H_{34}\}\ \sqcup\  
   \{H_{13},H_{24}\}\ \sqcup\  
   \{H_{14},H_{23}\}
\]
of the solutions.
We also see that $\calX^{(2)}$ has two components.
Exactly two sets $\calC_{(\emptyset,\emptyset,\emptyset,\emptyset),0}$ and 
$\calC_{(\{2\},\{2\},\{2\},\{2\}),2}$ have dimension $N(\blambda)$.

The structure of this problem shows that if the $K_i$ are real, then either two or
all six of the solutions will be real.
Indeed, if all four solutions $h_i$ to the auxiliary problem are real, than all six solutions
$H_{i,j}$ will also be real.
If however, two or four of the $h_i$ occur in complex conjugate pairs, then exactly two of
the $H_{i,j}$ will be real.

%
\subsection{The Schubert problem $\sTT^2\cdot\sIII\cdot\sTh\cdot\sI^2=4$.}

Let $\ell\in\Gr(2,8)$, $K_1,K_2,L_1,L_2\in\Gr(4,8)$, and $\Lambda\in\Gr(6,8)$ be general.
These give an instance of this Schubert problem,
 \begin{multline*}
  \Omega_{\ssTT}K_1\cap \Omega_{\ssTT}K_2 \cap \Omega_{\sIII}\Lambda\cap
  \Omega_{\sTh}\ell\cap \Omega_{\sI}L_1\cap \Omega_{\sI}L_2
   \ =\ \{ H\in\Gr(4,8)\,\mid\,\\ \dim H\cap K_i\geq 2\,,
   \dim H\cap\Lambda \geq 3\,,\ \dim H\cap\ell\geq 1\,,\ 
   \dim H\cap L_i\geq 1\,, \ \mbox{for }i=1,2\}\,.
 \end{multline*}
Any solution $H$ meets each of the 6-planes $\langle K_i,\ell\rangle$
in a 3-plane and therefore their four-dimensional intersection 
$\defcolor{M}:=\langle K_1,\ell\rangle\cap\langle K_2,\ell\rangle$ in a 2-plane, $h$.
Then $h$ must meet the 2-planes $\ell$, $M\cap K_1$, $M\cap K_2$, and $M\cap\Lambda$, so it is
a solution to the problem 
\[
  \Omega_{\sI}\ell  \,\cap\,
  \Omega_{\sI}(M\cap K_1)  \,\cap\,
  \Omega_{\sI}(M\cap K_2)  \,\cap\,
  \Omega_{\sI}(M\cap \Lambda)
\]
in $\Gr(2,M)$.
As the subspaces are in general position, this has two solutions $h_1$ and $h_2$.

Any solution $H$ to our original problem also meets each of the 2-planes
$\Lambda\cap K_1$ and $\Lambda\cap K_2$ in a 1-plane, and therefore
meets their span, $M'$, in a 2-plane, $m$.
As $M+M'=\C^8$, they are in direct sum and $H$ is the span of $m$ and one of the $h_i$.

Fix $i\in\{1,2\}$ and suppose that $h_i\subset H$.
Then $H$ meets each of $\langle h_i,L_j\rangle$ for $j=1,2$ in a 3-plane and therefore
$H$ meets each of the 2-planes $M'\cap\langle h_i,L_1\rangle$ and $M'\cap\langle
h_i,L_2\rangle$.
Thus $m=H\cap M'$ is a solution to the problem 
\[
  \Omega_{\sI}(\Lambda\cap K_1)  \,\cap\,
  \Omega_{\sI}(\Lambda\cap K_2)  \,\cap\,
  \Omega_{\sI}(M'\cap\langle h_i,L_1\rangle) \,\cap\,
  \Omega_{\sI}(M'\cap\langle h_i,L_2\rangle)
\]
in $\Gr(2,M')$.
This has two solutions, $m_{i,1}$ and $m_{i,2}$.

The four 4-planes
$\defcolor{H_{i,j}}:=\langle h_i,m_{i,j}\rangle$ for $i,j=1,2$ 
are the solutions to our Schubert problem.
Since each element of the Galois group 
either fixes $h_1$ and $h_2$ or it interchanges them, it preserves the partition 
\[
  \{H_{1,1},H_{1,2}\}\,\sqcup\,
  \{H_{2,1},H_{2,2}\}\,,
\]
and so it is imprimitive.
In fact, it is a subgroup of the dihedral group $D_4$ of symmetries
of the square whose diagonals are the partition.
We verified that it was the dihedral group $D_4$ by reducing modulo primes $p$ and 
computing the cycles type of the resulting Frobenius elements (this method is described
in~\cite[\S5.4]{MSJ}) and found elements of the Galois group with 
cycle types
\[ 
   (4)\,,\ 
   (2,2)\,,\ 
   (2,1,1)\,,\ \mbox{ and }\ (1,1,1,1)\,.
\]
As the only subgroup of $D_4\subset S_4$ having elements of these cycle types is $D_4$
itself, we conclude that the original Schubert problem had Galois group $D_4$.

It is an exercise to verify that the three problems
$\sTT^2\cdot\sIII\cdot\sThI\cdot\sI=4$, 
$\sTT^2\cdot\sTII\cdot\sTh\cdot\sI=4$, and 
$\sTT^2\cdot\sTII\cdot\sThI=4$, have nearly the same geometry behind their
solutions and also have Galois group $D_4$.
This last problem was mentioned by Billey and Vakil~\cite{BV}, who asked if its Galois group
was indeed $D_4$.

%
\subsection{The Schubert problem $\sIII^2\cdot\sTh^2\cdot\sI^4=4$}

An instance of this problem
is given by the choice of two 6-planes
$L_1,L_2$, two 2-planes $\ell_1,\ell_2$ and four 4-planes $K_1,\dotsc,K_4$, all in general
position. 
Solutions will be those $H\in\Gr(4,8)$ such that
 \begin{equation}\label{Eq:SP}
   \dim H\cap L_i\ \geq\ 3\,,\ 
   \dim H\cap \ell_i\ \geq\ 1\,,\ \mbox{and}\ 
   \dim H\cap K_j\ \geq\ 1\,,\ 
 \end{equation}
for $i=1,2$ and $j=1,\dotsc,4$.

Consider the first four conditions in~\eqref{Eq:SP}.
Let $\defcolor{\Lambda}:=\langle \ell_1,\ell_2\rangle$, the linear span of $\ell_1$ and
$\ell_2$, which is isomorphic to $\C^4$.
Then $\defcolor{h}:=H\cap\Lambda$ is two-dimensional.
If we set $\ell_3:=\Lambda\cap L_1$ and $\ell_4:=\Lambda\cap L_2$, then 
$\dim h\cap \ell_3=\dim h\cap \ell_4=1$, and so 
$h\in\Gr(2,\Lambda)\simeq \Gr(2,4)$ meets each of the four two-planes
$\ell_1,\dotsc,\ell_4$, so $h$ is a solution to the problem $\sI^4=2$ in $\Gr(2,\Lambda)$
given by $\ell_1,\dotsc,\ell_4$, and therefore there are
two solutions, $h_1$ and $h_2$.

Now set $\defcolor{\Lambda'}:=L_1\cap L_2$, which is four-dimensional, and fix
one of the solutions \defcolor{$h_a$} to the problem of the previous paragraph.
For each $j=1,\dotsc,4$, set $\defcolor{\mu_j}:=\langle h_a, K_j\rangle\cap\Lambda'$,
which is two-dimensional.
These four two-planes are in general position and therefore give an instance of $\sI^4=2$ in
$\Gr(2,\Lambda')$. 
Let \defcolor{$ m_{a,1}$} and \defcolor{$ m_{a,2}$} be the two solutions to this
problem, so that $\dim  m_{a,b}\cap\mu_j\geq 1$ for each $j$.

Then the four subspaces $\defcolor{H_{ab}}:=\langle h_a, m_{a,b}\rangle$ are solutions to
the original Schubert problem.
Indeed, since $\dim H_{ab}\cap\ell_j=1$ for $j=1,\dotsc,4$ and 
$ m_{a,b}\subset L_1\cap L_2$, we have $\dim H_{ab}\cap L_i=3$, and so $H_{ab}$
satisfies the first four conditions of~\eqref{Eq:SP}. 
Since $\Lambda\cap\Lambda'=\{0\}$, $h_a$ does not meet $\mu_j$ for $j=1,\dotsc,4$, so 
$\dim H_{ab}\cap\langle h_a,K_j\rangle=3$, which implies that $\dim H_{ab}\cap K_j\geq 1$,
and shows that $H_{ab}$ is a solution.

The Galois group preserves the partition 
$\{H_{11},H_{12}\}\sqcup\{H_{21},H_{22}\}$ and so it is imprimitive.
The same arguments as before show that $D_4$ is its Galois group.
It is also an exercise to show that each of the following eight Schubert problems has a
solution whose description is nearly identical to $\sIII^2\cdot\sTh^2\cdot\sI^4=4$,
and therefore has Galois group $D_4$.
\begin{center}
  $4\ =\ \sTII^2\cdot\sThI^2\ =\ 
   \sTII^2\cdot\sThI\cdot\sTh\cdot\sI\ =\ 
   \sTII\cdot\sIII\cdot\sThI^2\cdot\sI\ =\ 
   \sTII^2\cdot\sTh^2\cdot\sI^2$ \\
  $ =\ \sTII\cdot\sIII\cdot\sThI\cdot\sTh\cdot\sI^2\ =\ 
   \sIII^2\cdot\sThI^2\cdot\sI^2\ =\ 
  \sTII\cdot\sIII\cdot\sTh^2\cdot\sI^3\ =\ 
   \sIII^2\cdot\sThI\cdot\sTh\cdot\sI^3$\,.
\end{center}

\providecommand{\bysame}{\leavevmode\hbox to3em{\hrulefill}\thinspace}
\providecommand{\MR}{\relax\ifhmode\unskip\space\fi MR }
\providecommand{\MRhref}[2]{%
  \href{http://www.ams.org/mathscinet-getitem?mr=#1}{#2}
}
\providecommand{\href}[2]{#2}

\end{document}